\documentclass[11pt,leqno]{amsart}
\title[Involutive braidings: the quantum algebra $\mathfrak{gl}_{k,m}$]
{Non-combinatorial involutive braidings: the quantum algebra $\mathfrak{gl}_{k,m}$}
\author[Anastasia Doikou]{Anastasia Doikou}

\address[Anastasia Doikou] {Department of Mathematics, Heriot-Watt University,
Edinburgh EH14 4AS $\&$ Maxwell Institute for Mathematical Sciences, Edinburgh EH8 9BT, UK}
\email{a.doikou@hw.ac.uk}

\usepackage[utf8]{inputenc}
\usepackage{mathtools}
\usepackage{ytableau}
\usepackage{amsmath,amssymb}
\usepackage{tikz}
\usetikzlibrary{decorations.pathreplacing}
\usepackage{wrapfig}
\usepackage{lscape}
\usepackage{rotating}
\usepackage{epstopdf}
\usepackage[colorlinks,linkcolor=blue]{hyperref}
\usepackage{lipsum}
\usepackage{emptypage}

  \usepackage{latexsym} 
  \usepackage[all]{xy}
  \usepackage{amsfonts} 
  \usepackage{amsthm} 
  \usepackage{amsmath} 
  \usepackage{amssymb}
  \usepackage{pifont}  
  \usepackage{enumerate}
  \usepackage{dcolumn}
  \usepackage{comment}
  \usepackage{hyperref}  
\usepackage[english]{babel}
  \usepackage{amsfonts} 
\usepackage[utf8]{inputenc}
\usepackage{graphicx}
\usepackage{subcaption}
\usepackage[colorinlistoftodos]{todonotes}
\usepackage{indentfirst}
\usepackage{pifont}  
  \usepackage{enumerate}
  \usepackage{dcolumn}
  \usepackage{comment}
\usepackage[capitalise]{cleveref}
  \usepackage[capitalise]{cleveref}
\usepackage{multirow}
\usepackage{tikz}
\usetikzlibrary{automata, positioning, arrows}

\usetikzlibrary{trees,arrows}

  \usepackage{latexsym} 
  \usepackage[all]{xy}
  \usepackage{amsfonts} 
  \usepackage{amsthm} 
  \usepackage{amsmath} 
  \usepackage{amssymb}
  \usepackage{pifont}  
  \usepackage{enumerate}
  \usepackage{dcolumn}
  \usepackage{comment}
  \usepackage{hyperref}  
  \usepackage{tikz-cd}
\usepackage[english]{babel}
  \usepackage{amsfonts} 
\usepackage[utf8]{inputenc}

 \usepackage{tikz}
\usetikzlibrary{arrows}

 \newcolumntype{2}{D{.}{}{2.0}}
  \xyoption{2cell}

\newcommand{\hiddenpower}[2] { \ifnum \numexpr#2=1 #1 \else #1^#2 \fi }
\numberwithin{equation}{section}

\def\be{\begin{equation}}
\def\ee{\end{equation}}
\def\ba{\begin{eqnarray}}
\def\ea{\end{eqnarray}}

\newcommand{\cal}{\mathcal}

\usepackage{xparse}
\ExplSyntaxOn
\newcounter{diff_order}
\newcounter{diff_power}

\newcommand{\rawdiff}[3]
{
	\setcounter{diff_order}{0}
	\clist_map_inline:nn{#3}{\stepcounter{diff_order}}
	
	\frac{\hiddenpower{#1}{\thediff_order} #2}
	{
		\def\old_var{DefaultValue}
		\setcounter{diff_power}{0}
		
		\clist_map_inline:nn{#3}
		{
			\def\new_var{##1}
			\ifnum \thediff_power=0
				\stepcounter{diff_power}
			\else
				\tl_if_eq:NNTF \new_var \old_var
				{\stepcounter{diff_power}}
				{
					#1 \hiddenpower{\old_var}{\thediff_power}
					\setcounter{diff_power}{1}
				}
			\fi

			\def\old_var{##1}
		}
		
		#1 \hiddenpower{\old_var}{\thediff_power}
	}
}
\def\Label#1{\label{#1}\ifmmode\llap{[#1] }\else 
  \marginpar{\smash{\hbox{\tiny [#1]}}}\fi} 
  \def\Label{\label} 

\ExplSyntaxOff

\newlength{\bibitemsep}
\newlength{\bibparskip}
\let\oldthebibliography\thebibliography
\renewcommand\thebibliography[1]{%
  \oldthebibliography{#1}%
  \setlength{\parskip}{\bibitemsep}%
  \setlength{\itemsep}{\bibparskip}%
}

\newtheorem{thm}{Theorem}[section]
\newtheorem{lemma}[thm]{Lemma}
\newtheorem{cor}[thm]{Corollary}
\newtheorem{pro}[thm]{Proposition}
\newtheorem{defn}[thm]{Definition}

\newtheorem{rem}[thm]{Remark}
\newtheorem{exa}[thm]{Example}

\newcommand{\id}{\operatorname{id}}

\DeclareMathOperator{\EEnd}{End}

\newenvironment{wideregion}[1][0mm]
    {\vspace{#1}\list{}{\leftmargin-2cm\rightmargin\leftmargin}\item\relax}
    {\endlist}

\newenvironment{widegather }{\wideregion[-9mm]\gather}{\endgather\endwideregion}

\begin{document}

\hfill
 \begin{abstract} 
We investigate involutive, non-combinatorial solutions of the braid equation, viewing them as special deformations of the permutation map. Utilizing these solutions, we identify the associated quantum algebra and introduce it as the $\mathfrak{gl}_{k,m}$
Yangian. This newly derived Yangian is distinct from the known Yangian of the general linear Lie superalgebra; crucially, as a Hopf algebra, it possesses the standard tensor product algebra structure. The underlying algebra $\mathfrak{gl}_{k,m}$
is also introduced as a novel structure and constitutes a subalgebra of the Yangian. We then construct specific highest-weight modules of $\mathfrak{gl}_{k,m}$
that simultaneously yield the eigenstates of spin-chain-like ``Hamiltonians'', which are defined as the sum of the generators of the
$A$-type braid group. Furthermore, we study the highest-weight representations and the corresponding combinatorial bases for
$\mathfrak{gl}_{1,1}$, explicitly linking them to specific shapes of Young tableaux.

\end{abstract}
\maketitle

\date{}




\section{Introduction}

\noindent 
The study of non-commutative algebras has long placed the quantum Yang–Baxter equation \cite{Baxter, Yang} at the crossroads of representation theory \cite{FultonHarris}, low-dimensional topology \cite{Kassel}, quantum integrability and quantum algebras \cite{FRT,Jimbo1, Jimbo2, Drinfeld}. 
Involutive, combinatorial solutions to the braid equation \cite{Baxter, Yang} have been widely studied (see, for example, \cite{Eti, Gateva, Jespers, Ru05, Ru07})--particularly since Drinfel’d proposed the set-theoretic version 
\cite{Drin} of the Yang-Baxter equation. By contrast, in the present investigation we focus on a class of non-combinatorial, non-parametric, involutive solutions of the braid equation (\ref{braid}) (see also \cite{Lechner}) and construct the quantum universal enveloping algebras--specifically of Yangian type--associated with these non-combinatorial solutions. To achieve this we employ the  Faddeev-Reshetikhin-Takhtajan (FRT) framework \cite{FRT}, which provides a systematic means to extract a quadratic quantum algebra from a spectral-parameter-dependent solution of the Yang–Baxter equation. Spectral-parameter-dependent solutions are obtained from  solutions of the braid equation via the Baxterization procedure \cite{Jones}. 

Specifically, in Section 2, following the Baxterization of our non-combinatorial solutions,  we employ the FRT formulation \cite{FRT} to derive the associated quantum algebra (Theorem \ref{alg1}). This algebra is
referred to as the \(\mathfrak{gl}_{k,m}\) Yangian, denoted by \(Y(\mathfrak{gl}_{k,m})\) and
naturally possesses a  Hopf algebra structure (Proposition \ref{hopf}).
The Yangian \(Y(\mathfrak{gl}_{k,m})\) contains an underlying finite subalgebra, \(\mathfrak{gl}_{k,m}\), which itself constitutes a novel non-commutative algebraic structure. We provide a comprehensive analysis of \(\mathfrak{gl}_{k,m}\) by explicitly deriving its quadratic relations and defining its Chevalley–Serre type relations (Corollaries \ref{remg} and \ref{serre})--which are  distinct from those of the standard general linear Lie superalgebra \cite{Kac}--and constructing its fundamental Casimir invariants (Proposition \ref{symm}). Notably, elements of $ \mathfrak{gl}_{k,m} $ satisfy nilpotency conditions, which serve as a signature of supersymmetry, allowing $\mathfrak{gl}_{k,m}$ to operate as a super-algebra. Crucially, the newly derived universal algebra $Y(\mathfrak{gl}_{k,m})$ is distinct from the standard, extensively studied Yangian of the general linear Lie superalgebra \cite{Molev, Nazarov, Zhang}. From a structural perspective, a key distinction lies in the underlying tensor structure: as a Hopf algebra, \(Y(\mathfrak{gl}_{k,m})\) is equipped with the standard, ungraded tensor product algebra structure, such that \(({\mathrm a} \otimes {\mathrm b})({\mathrm c} \otimes {\mathrm d}) = {\mathrm a}{\mathrm c} \otimes {\mathrm b}{\mathrm d}\), for \({\mathrm a}, {\mathrm b}, {\mathrm c}, {\mathrm d} \in Y(\mathfrak{gl}_{k,m})\), thereby circumventing the standard \(\mathbb{Z}_{2}\)-graded super-commutativity signs. This requires deforming the algebra coproducts, as shown in Proposition \ref{hopf}, and introducing a family of additional group elements into the algebra, both of which are dictated by the FRT formulation. Consequently, this framework provides a surprising mechanism to generate super-type symmetries entirely through coproduct deformation, leaving the underlying tensor multiplication completely ungraded.
 
 The primary novelty of our approach lies in treating both the involutive braid solution and its associated quantum universal enveloping algebra as non-parametric deformations of the  permutation (or flip) map and the classical $\mathfrak{gl}_{k+m}$ Yangian \cite{yangians}, respectively. Within this framework, the emerging supersymmetry is elegantly encoded--and effectively hidden--directly in the deformation of the permutation operator.  It is worth noting that our construction is conceptually parallel to the classic Radford–Majid bosonization framework \cite{Radford, Majid2}. 
 While standard bosonization introduces a bicrossproduct to map a braided structure to an ordinary Hopf algebra, our framework departs from this paradigm entirely. By using our braiding--a non-parametric deformation of the permutation map-- we retain a strictly ungraded tensor product multiplication, but absorb the super-algebraic data entirely into a coassociative coproduct and a commutative family of group elements (Proposition \ref{hopf}). This provides a clean algebraic framework that completely bypasses the computational burden of explicit \(\mathbb{Z}_{2}\)-graded signs.

In Section 3, centralizers of the $A$-type Artin braid group are  identified as coproducts of representations of the algebra \(\mathfrak{gl}_{k,m}\), allowing us to construct distinct combinatorial bases for finite-dimensional irreducible representations of the algebra (Theorem \ref{basic12}). These bases also provide eigenstates of operators, known as quantum spin-chain-like ``Hamiltonians'', constructed from the sum of all generators of the \(A\)-type Artin braid group. For instance, in the special case of the algebra $\mathfrak{gl}_{1,1}$ the quantum Hamiltonian reduces to a variant of the Heisenberg XX model \cite{XX}. 
In Section 4, we provide an explicit classification for \(\mathfrak{gl}_{1,1}\), proving that its highest-weight modules are inherently two-dimensional. Furthermore, we construct an explicit bijection between these combinatorial bases and specific hook-like shapes of super-symmetric semi-standard Young tableaux \cite{BR83, Mac92}, offering a concrete combinatorial realization of the representations.

\subsection{Preliminaries}
Prior to presenting the main analysis, we review essential preliminaries concerning involutive solutions to the braid equation and the process of Baxterization \cite{Jones}.

\noindent  Formally,  given a vector space $V$, a solution to the Yang-Baxter equation in its braided form is a linear map
\(\check r: V \otimes V \to V \otimes V\) satisfying the relation:
\begin{equation}\label{braid}
\left(\check r\otimes\id_V\right)
\left(\id_V\otimes\check r\right)
\left(\check r\otimes\id_V\right)
= 
\left(\id_V\otimes\check r\right)
\left(\check r\otimes\id_V\right)
\left(\id_V\otimes\check r\right). 
\end{equation}
If $\check r$ is a solution such that $\check r^2=\id_{V \otimes V}$, then $\check r$ is said to be \emph{involutive}.
We  also recall the connection between the braid equation (\ref{braid}) and the 
Yang-Baxter equation in its more standard form \cite{Jimbo1, Jimbo2}. 
We introduce the linear map $r: V\otimes V\rightarrow V\otimes V,$ such that $r = {\mathrm P} \check r,$ 
where ${\mathrm P}: V\otimes V\rightarrow V\otimes V$ 
is the permutation or flip map: ${\mathrm P}(u\otimes w) = w \otimes u,$  $u,w \in V.$ Hence, $r: V \otimes V \to V \otimes V$  
satisfies the Yang-Baxter equation \cite{Baxter, Yang}:
\begin{equation}
r_{12}\ r_{13}\  r_{23} = r_{23}\ r_{13}\ r_{12},  \label{YBE}
\end{equation} 
where if $\ r = \underset{j}{\sum} a_j \otimes b_j,$ $a_j, b_j \in \mbox{End}(V),$ we then denote $\ r_{12} =\underset{j}{\sum} a_j \otimes b_j \otimes \id_V,$  $r_{23} =\underset{j}{\sum}  \id_V \otimes a_j \otimes b_j$ 
and $r_{13} =\sum_j a_j \otimes  \id_V \otimes b_j$ ({\it index notation}). If $\check r$ is involutive then $r$ satisfies $r_{12} r_{21} = \mbox{id}_{V\otimes V}$ and is called {\it reversible}.

It is also useful to present the precise definition of involutive, combinatorial and non-combinatorial solutions of the braid equation.
\begin{defn} \label{proto}
    Let $X = \big \{x_1, x_2, \ldots,x_n \big \}$ and $\sigma_x,\ \tau_y: X \to X,$ such that $y \mapsto \sigma_x(y)$ and $x \mapsto \tau_y(x).$ Let also $\big \{ e_x \big \}_{x\in X}$ be the standard canonical basis of ${\mathbb C}^n$, and  $\check r: {\mathbb C}^n \otimes {\mathbb C}^n \to {\mathbb C}^n \otimes {\mathbb C}^n,$ such that $\check r (e_x \otimes e_y )= \underset{z,w \in X}{\sum}{\mathfrak c}_{z,x|w,y} e_{z} \otimes e_w,$ ${\mathrm c}_{z,x|w,y} \in {\mathbb C},$ for all $x,y,z,w \in X,$ be an involutive solution of the braid equation. Then $\check r$ is said to be: 
    
    \begin{enumerate}

        \item an involutive \underline{set-theoretic or combinatorial} solution of the braid equation if\\ ${\mathfrak c}_{z,x| w,y} = \delta_{z,\sigma_{x}(y)} \delta_{w,\tau_{y}(x)}.$ 

        \item an involutive \underline{non-combinatorial} solution of the braid equation if ${\mathfrak c}_{z,x| w,y} \neq \delta_{z,\sigma_{x}(y)} \delta_{w,\tau_{y}(x)}.$
        \end{enumerate}
\end{defn}
Relevant useful definitions on combinatorial and non-combinatorial maps are given in \cite{mybraided}; see also \cite{DoiSmo} for related expressions of combinatorial solutions.
\begin{exa} Two simple examples of 
involutive, combinatorial solutions of the braid equation are given below:
\begin{enumerate}
\item The permutation or flip map: $\check r (e_x \otimes e_y) = e_y \otimes e_x,$ $x, y \in \big \{1,2, \ldots,  n\big \}.$

\item The Lyubashenko solution \cite{Drin}: $\check r (e_x \otimes e_y) = e_{y+1} \otimes e_{x-1},$ $x,y \in \big \{1,2, \ldots, n\big \},$
where the addition (subtraction) is defined $\mod n.$
\end{enumerate}
\end{exa}

\subsection*{The braid group}
We recall the Artin presentation of the braid group, i.e. 
the standard braid group on $N$ strands.
\begin{defn}  \label{Artin}
The $A$-type Artin braid group $B_{N}$ is defined by generators $\sigma_1, \sigma_2, \ldots, \sigma_{N-1}$ and relations
\begin{equation}
\sigma_i\sigma_{i+1}\sigma_i = \sigma_{i+1} \sigma_i \sigma_{i+1}, ~~ \mbox{and }~~\sigma_i\sigma_j=\sigma_j\sigma_i~~\mbox{if}~~|i-j|>1. \nonumber
\end{equation}
\end{defn}
Every braid on $N$ strands determines a permutation on $N$ elements. This assignment becomes a map $B_N \to S_N,$ 
such that $\sigma_i \in B_N$ is mapped to the transposition $s_i = (i, i+1) \in S_N$. 
These transpositions generate the symmetric group, satisfy the braid group relations and in addition  $s_i^2 =1.$ This transforms the Artin presentation of the braid group into the Coxeter presentation of the symmetric group. Note that in general $\sigma_i^2 \neq 1.$

We focus here on involutive tensor representations of the braid group. 
Specifically, let $\rho: B_N \to \EEnd(V^{\otimes N}),$ such that $\sigma_j \mapsto \check r_j,$ $j \in [N-1]$ where
\begin{equation} 
\check r_j := \id_V^{\otimes (j-1)}  \otimes  \check r \otimes \id_V^{\otimes (N-j-1)}, \label{tensor}
\end{equation}
$ \check r \in \EEnd(V \otimes V)$ is an involutive solution of the braid equation, i.e. $\check r$ satisfies the braid identity and $\check r^2 = \id.$ 


\subsection*{Involutive solutions and Baxterization}

\noindent 
We recall the Baxterization \cite{Jones} of involutive solutions of the 
braid and Yang-Baxter equations. 

We first recall the braid equation in the presence of spectral parameters 
$\lambda_1,\ \lambda_2 \in {\mathbb C}$ ($\delta = \lambda_1 - \lambda_2$) reads as \cite{Jones}:
\begin{equation}
\check R_{12}(\delta)\ \check R_{23}(\lambda_1)\ \check R_{12}(\lambda_2) = \check R_{23}(\lambda_2)\
 \check R_{12}(\lambda_1)\ \check R_{23}(\delta), \label{YBE1}
\end{equation}
where $\check R: V \otimes V\to V \otimes V.$

We focus on solutions of (\ref{YBE1}), obtained by Baxterization of involutive solutions $\check r$, and being of the form $~\check R(\lambda) = \lambda \check r + \id_V^{\otimes 2},$ $\lambda \in {\mathbb C}.$ Let also $R = {\mathrm P} \check R$, then
\begin{equation}
R(\lambda)= \lambda r + {\mathrm P}, \label{braid2}
\end{equation}
and $R:V\otimes V \to V\otimes V$ is a solution of the parametric Yang-Baxter equation \cite{Baxter, Yang},
\begin{equation}
 R_{12}(\delta)\  R_{13}(\lambda_1)\  R_{23}(\lambda_2) = R_{23}(\lambda_2)\ R_{13}(\lambda_1)\ R_{12}(\delta). \label{YBE2}
\end{equation}

Moreover, it follows that $R_{12}(\lambda)\  R_{21}(-\lambda) = (-\lambda^2 +1)  \id_V^{\otimes 2},$ then $R$ is said to be unitary.


\section{The quantum algebra \texorpdfstring{$\mathfrak{gl}_{k,m}$}{gl} }

\noindent We now examine a special class of involutive solutions of the braid equation, which can be viewed as non-parametric diagonal deformations of the permutation map; these are derived in Proposition \ref{pro11} below (see also \cite{Lechner}). Before proceeding with our analysis, we introduce the notation used throughout the manuscript.
The set of all non-negative integers is denoted ${\mathbb N},$ the set of all positive integers is denoted ${\mathbb N}^*,$ $[n] := \big \{ 1,2, \ldots, n \big \}$ and $X_k^+:=\big \{k+1, k+2, \dots, n \big \},$ $k < n.$
Recall also that $\big \{ e_x\big \}_{x \in [n]}$ is the standard canonical basis of ${\mathbb C}^n,$ and the elementary $n\times n$ matrices $e_{x,y}: = e_x e_y^T$ ($^T$ denotes transposition), such that $(e_{x,y})_{z,w}= \delta_{x,z}\delta_{y,w},$ $x,y,z,w\in [n].$ $\big \{e_{x,y}\big\}_{x,y \in [n]}$ is a basis of $\mbox{End}({\mathbb C}^n).$ 

\begin{pro} \label{pro11} Let ${\mathrm P}: {\mathbb C}^n \otimes {\mathbb C}^n  \to {\mathbb C}^n \otimes {\mathbb C}^n $ be the permutation map explicitly expressed as $ {\mathrm P} = \underset{x,y \in [n]}{\sum} e_{x,y} \otimes e_{y,x}.$ Let also ${\mathrm D}= \id- 2\alpha \underset{x \in X_k^+ }{\sum} e_{x,x} \otimes e_{x,x},$ then $\check r = {\mathrm D} {\mathrm P}$ is an involutive solution of the braid equation (\ref{braid}) if and only if $\alpha=0$ or $\alpha =1.$
\begin{proof}
The proof is simple by first showing that ${\mathrm D}^2 =\id + 4\alpha(\alpha-1)\underset{x\in X_k^+}{\sum}  e_{x,x} \otimes e_{x,x}$ and  ${\mathrm P}{\mathrm D} = {\mathrm D} {\mathrm P} = {\mathrm P} - 2 \alpha \underset{x\in X_k^{+}}{\sum} e_{x,x }\otimes e_{x,x}.$ 
The involutivity condition is satisfied if and only if ${\mathrm D}^2 =\id,$ i.e. $\alpha =0$ or $\alpha=1.$  Using also the fact that ${\mathrm P}$ satisfies the braid equation and ${\mathrm P} (A \otimes B) {\mathrm P} = B \otimes A,$ for $A,\ B \in \mbox{End}({\mathbb C}^n),$  it is immediately shown that ${\mathrm D}{\mathrm P}$ is also a solution of the braid equation.
\end{proof}
\end{pro}
 
Let $B =\big \{e_x\big \}_{x\in [n]}$ be the standard canonical basis of ${\mathbb C}^{n},$ then the action of the $\check r$-matrix of Proposition \ref{pro11} on $B^{ \otimes 2}$ is given by ($\alpha \in \big \{0,1\big \}$)
\begin{equation}
    \check r(e_{x} \otimes  e_{y})= \Bigg  \{ \begin{matrix} & e_{x} \otimes e_{x}, ~~~~~~~~~~~~~~~x=y \in [k],  \\ 
 & (1-2\alpha)e_{x} \otimes e_{x},~~~~x = y \in X_k^+ \label{action01}
 \\ & e_{y} \otimes  e_{x}, ~~~~~~~~~~~~~~~x\neq y \in [n]. \end{matrix},\end{equation}
 Notice that if $\alpha =0$ one recovers the flip map (combinatorial solution), whereas if $\alpha=1$ the solution is non-combinatorial.
The action of the elements $\check r_j,$ $j \in [N-1]$ on $B^{\otimes N}$ is given by (see also (\ref{tensor})) 
\begin{equation}
    \check r_j(e_{x_1} \otimes \ldots e_{x_j} \otimes e_{x_{j+1}} \otimes \ldots e_{x_N})= \Bigg  \{ \begin{matrix} & e_{x_1} \otimes \ldots e_{x_j} \otimes e_{x_{j} }\otimes \ldots e_{x_N}, ~~~~~~~~~~~~~~~x_j=x_{j+1} \in [k],  \\ 
 & (1-2\alpha) e_{x_1} \otimes \ldots e_{x_j} \otimes e_{x_j} \otimes \ldots e_{x_N},~~~~x_j = x_{j+1} \in X_k^+ \label{action1}
 \\ & e_{x_1} \otimes \ldots e_{x_{j+1}} \otimes e_{x_j} \otimes \ldots  e_{x_N}, ~~~~~~~~~~~~x_j\neq x_{j+1} \in [n]. \end{matrix},\end{equation}
 
We recall the definition of quadratic algebras \(\mathfrak{A}\) associated with solutions to the Yang–Baxter equation, which emerge from the FRT construction \cite{FRT}. Specifically, a quadratic algebra can be derived from any given solution of the Yang–Baxter equation. As a primary example, the Yangian of any classical Lie algebra is examined in \cite{yangians}, whereas in \cite{Nazarov} the Yangians of Lie super-algebras are studied.
\begin{defn} \label{frt1}
Let $R(\lambda)\in \EEnd({\mathbb C}^n\otimes {\mathbb C}^n)$ be a solution of the Yang-Baxter equation (\ref{YBE2}), $\lambda \in {\mathbb C}.$ Let also $L(\lambda)  := \underset{x,y\in [n]}{\sum} e_{x,y} \otimes L_{x,y}(\lambda) \in  \EEnd(\mathbb{C}^n) \otimes {\mathfrak A},$ $L_{x,y}(\lambda) =\underset{p\in {\mathbb N} }{\sum}\lambda^{-p} L^{(p)}_{x,y}\in {\mathfrak A},$ where ${\mathfrak A}$ is the quantum algebra associated to the solution $R,$ and is defined as the quotient of the free unital, associative ${\mathbb C}$-algebra, generated by indeterminates $\Big \{ L^{(p)}_{x,y}|\ x,y \in [n],\ p\in {\mathbb N}\Big\}$ and relations 
\begin{equation}
R_{12}(\lambda_1, \lambda_2)\ L_1(\lambda_1)\ L_2(\lambda_2) = L_2(\lambda_2)\ 
L_1(\lambda_1)\  R_{12}(\lambda_1,  \lambda_2), \label{RTT}
\end{equation}
where $\lambda_1, \lambda_2 \in {\mathbb C},$ $R_{12} =R \otimes  1_{\mathfrak A },$  
$L_{1}=\underset{{x,y \in [n] }}{\sum}e_{x,y}\otimes \id_n \otimes  L_{x,y}$\footnote{Notice that in $L$ in addition to the indices 1 and 2 in (\ref{RTT}) there is also an implicit ``quantum index'' $3$ associated to ${\mathfrak A},$ 
which for now is omitted, i.e. one writes $L_{13},\ L_{23}$.}, $L_{2}=\underset{x,y \in [n]}{\sum} \id_n \otimes  e_{x,y}\otimes  L_{x,y}$ ($\id_n$ denotes the $n \times n$ identity matrix). 
\end{defn}
Note that if equation (\ref{RTT}) holds, then $R$ is necessarily a solution of the Yang-Baxter equation (\ref{YBE2}) (see for instance \cite{majid}). Essentially, definition \ref{frt1} states that different choices of solutions of the Yang-Baxter equation yield distinct quantum algebras. This formulation is also known as the RTT presentation of a quantum algebra. We emphasize once more that here we retain strictly ungraded tensor products, i.e. $(e_{x,y} \otimes e_{z,w}) (e_{x',y'} \otimes e_{z',w'}) = e_{x,y} e_{x',y'} \otimes e_{z,w} e_{z',w'},$ for all $x,y,z,w,x',y',z',w' \in [n],$ $e_{x,y} \in \mbox{End}({\mathbb C}^n),$ and $({\mathrm a} \otimes {\mathrm b})({\mathrm c} \otimes {\mathrm d}) = {\mathrm a}{\mathrm c} \otimes {\mathrm b}{\mathrm d}$, for ${\mathrm a}, {\mathrm b}, {\mathrm c}, {\mathrm d} \in {\mathfrak A}.$ 

Subsequently, we consider the involutive solution to the braid equation from Proposition \ref{pro11}. We retain the parameter \(\alpha \in \big \{0, 1 \big \}\) to track the deformation at \(\alpha = 1\), comparing it to the undeformed case (\(\alpha = 0\)), which corresponds to the \(\mathfrak{gl}_{n}\) Yangian \cite{yangians}.
\begin{thm} \label{alg1} Let $r$ be a solution of the parameter-free Yang-Baxter equation and $R(\lambda) = \lambda r + {\mathrm P}$ be the Baxterized solution (\ref{braid2}) (recall ${\mathrm P}$ is the permutation map). Let also  $L(\lambda)  := \underset{x,y\in [n]}{\sum} e_{x,y} \otimes L_{x,y}(\lambda) \in  \EEnd(\mathbb{C}^n) \otimes {\mathfrak A},$ where $L_{x,y}(\lambda) =\underset{p\in {\mathbb N} }{\sum}\lambda^{-p} L^{(p)}_{x,y}\in {\mathfrak A}$  and  $L^{(p)} = \underset{x \in [n]}{\sum} e_{x,y} \otimes L^{(p)}_{x,y}$ for all $p\in {\mathbb N}.$ Consider $r = \id - 2 \alpha \underset{ x \in X_k^+}{\sum}e_{x,x} \otimes e_{x,x},$ $~k+m =n,$ $~\alpha \in \big \{0,1\big \}.$ Also:
\begin{enumerate}
\item  If $\alpha =0,$ consider $L^{(0)}_{x,y} = \delta_{x,y} h_x,$  $h_x =1$ for all $x \in [n].$ 
\item If $\alpha =1,$ consider $L^{(0)}_{x,y} = \delta_{x,y} h_x,$  $h_x =1$ for all $x \in [k],$ $h_x$ 
is an invertible element with $h_x \neq 1,$ for all $x \in X_k^+.$  
\end{enumerate}
Then, the associated quantum algebra, denoted $Y(\mathfrak{g}_{\alpha}(k,m)),$ $\alpha \in \big \{0,1 \big \},$
is generated by indeterminates  $L^{(p)}_{x,y},$ $x,y \in [n],$ $p \in {\mathbb N}$ and relations for all $x,y,z \in [n],$ $l,p \in {\mathbb N}$:
\begin{eqnarray}
 &&\alpha=1:~~~ \big [h_x,\ h_y\big ]=0,  x,y \in X_k^+, \nonumber \\ &&\qquad \qquad ~\big [L^{(p)}_{x,y}, h_z  \big ] = 2L_{x,y}^{(p)}h_x \delta_{z,x} -2h_y L_{x,y}^{(p)} \delta_{z,y}, ~~~z\in X_k^{+}
\nonumber \\
&& \alpha \in \big \{0,1\big \}: ~~ \big[ L^{(p+1)}_{x,y},\ L^{(l)}_{z,w} \big ]  -\big[ L^{(p)}_{x,y},\ L^{(l+1)}_{z,w} \big ]-2\alpha \big ( L^{(p+1)}_{x,y}L^{(l)}_{x,w} -
L_{x,y}^{(p)}L^{(l+1)}_{x,w}\big )\delta_{x,z}\big|_{x \in X_k^+}  \nonumber\\ && \qquad \qquad ~~~~+2\alpha \big (L_{z,y}^{(l)} L^{(p+1)}_{x,y} -L_{z,y}^{(l+1)} L^{(p)}_{x,y}\big )\delta_{y,w} \big |_{y\in X_k^+}= L^{(l)}_{z,y}L^{(p)}_{x,w}  - L^{(p)}_{z,y} L^{(l)}_{x,w} \label{funda2}
\end{eqnarray}
where $[\ ,\ ]: Y(\mathfrak{g}_{\alpha}(k,m) )\times Y(\mathfrak{g}_{\alpha}(k,m))  \to Y(\mathfrak{g}_{\alpha}(k,m)),$ such that $[{\mathrm a},\ {\mathrm b}] = {\mathrm a}{\mathrm b}-{\mathrm b}{\mathrm a},$ for all ${\mathrm a},{\mathrm b} \in Y(\mathfrak{g}_{\alpha}(k,m) ).$
\end{thm}
\begin{proof}
The proof is based on the fundamental relation (\ref{RTT}) and the form of the Baxterized $R$-matrix. 
First recall, $R_{12}(\lambda) = R(\lambda) \otimes 1_Y,$ and
\begin{eqnarray}
&& L_1(\lambda) = \sum_{z, w \in [n]} e_{z,w} \otimes \id_n \otimes L_{z,w}(\lambda),\ \quad  L_2(\lambda)= \sum_{z, w \in [n]}
\id_n  \otimes  e_{z,w}  \otimes L_{z,w}(\lambda).  \label{def}
\end{eqnarray} 
The exchange relations among the various generators of the associated quantum algebra 
are derived via (\ref{RTT}) and after expressing $L$ as a formal power series expansion 
$L(\lambda) = \underset{p\geq 0}{\sum} {L^{(p)}\lambda^{-p}}$.
Substituting expression (\ref{braid2}), and the $\lambda^{-1}$ expansion of $ L(\lambda)$ in (\ref{RTT}) and focusing on the coefficients of the terms $\lambda_1^{-p} \lambda_2^{-l}$ we arrive at
the following relations of  associated 
to the $R$-matrix, $R(\lambda) = \lambda r +{\mathrm P},$ $\lambda \in {\mathbb C},$ for all $l,p \in {\mathbb N}^*,$ recall also $\check r = {\mathrm P}r,$
\begin{eqnarray}
& & \check r_{12} L^{(0)}_{1} L^{(0)}_{2} = L^{(0)}_{1} L_2^{(0)} \check r_{12}, ~~~ \check r_{12} L^{(0)}_{1} L^{(p)}_{2} = L^{(p)}_{1} L_2^{(0)} \check r_{12}, \nonumber \\
& &\check r_{12} L_{1}^{(p+1)} L_2^{(l)} -\check  r_{12} L_1^{(p)} L_2^{(l+1)} +  L_1^{(p)} L_2^{(l)}  = L_1^{(l)} L_{2}^{(p+1)} \check r_{12} -  L_1^{(l+1)} L_2^{(p)}\check r_{12} +  L_1^{(l)} L_2^{(p)}.
\label{funda}
\end{eqnarray}

We now focus on $\check r = {\mathrm P} - 2\alpha \underset{x\in X_k^+}{\sum}e_{x,x} \otimes e_{x,x},$ and recall  from (\ref{def}),
that $L^{(p)}_{x,y},$ $x,y \in [n]$ are the generators of the associated quantum algebra and in this case $L^{(0)}_{x,y} = \delta_{x,y}h_x,$ also $\check r_{12} = \check r \otimes  \mbox{id}_{\mathfrak A }.$ 
Substituting the above expressions in (\ref{funda}) we arrive at the defining relations (\ref{funda2}). Specifically, for all $l,p \in {\mathbb N}$ and $\alpha \in \big \{0, 1  \big \},$
\begin{eqnarray}
& & \check r_{12} L^{(0)}_{1} L^{(0)}_{2} = L^{(0)}_{1} L_2^{(0)} \check r_{12}\ \Rightarrow\ \big [h_x,\ h_y \big ] =0, ~~~~x,y \in X_k^+\nonumber\\ 
&& \check r_{12} L^{(p)}_{1} L^{(0)}_{2} = L^{(0)}_{1} L_2^{(p)} \check r_{12}\ \Rightarrow\ \big [L^{(p)}_{x,y}, h_z  \big ] = 2\alpha \big (L_{x,y}^{(p)}h_x \delta_{z,x} -h_y L_{x,y}^{(p)} \delta_{z,y} \big ), ~~~z\in X_k^{+}
\end{eqnarray}
and
\begin{eqnarray}
& &\check r_{12} L_{1}^{(p+1)} L_2^{(l)} -\check  r_{12} L_1^{(p)} L_2^{(l+1)} +  L_1^{(p)} L_2^{(l)}  = L_1^{(l)} L_{2}^{(p+1)} \check r_{12} -  L_1^{(l+1)} L_2^{(p)}\check r_{12} +  L_1^{(l)} L_2^{(p)}\ \Rightarrow \nonumber\\
&& \big[ L^{(p+1)}_{x,y},\ L^{(l)}_{z,w} \big ]  -\big[ L^{(p)}_{x,y},\ L^{(l+1)}_{z,w} \big ]-2\alpha \big ( L^{(p+1)}_{x,y}L^{(l)}_{x,w} -
L_{x,y}^{(p)}L^{(l+1)}_{x,w}\big )\delta_{x,z}\big|_{x \in X_k^+}  \nonumber\\ && +2\alpha \big (L_{z,y}^{(l)} L^{(p+1)}_{x,y} -L_{z,y}^{(l+1)} L^{(p)}_{x,y}\big )\delta_{y,w} \big |_{y\in X_k^+}= L^{(l)}_{z,y}L^{(p)}_{x,w}  - L^{(p)}_{z,y} L^{(l)}_{x,w}. \label{funda3}
\end{eqnarray}
\end{proof}
If  $\alpha =0$ the quantum algebra is the familiar $\mathfrak{gl}_{n}$ Yangian, i.e. $Y(\mathfrak{g}_{0}(k,m) )  = :Y(\mathfrak{gl}_{n})$ (recall $k+m=n$) \cite{yangians}.
If $\alpha =1,$ the associated quantum algebra is called the $\mathfrak{gl}_{k,m}$ Yangian and is denoted $Y(\mathfrak{gl}_{k,m})$ (i.e. $Y(\mathfrak{g}_{1}(k,m) ) =: Y(\mathfrak{gl}_{k,m})$). 
We focus henceforth on the case $\alpha =1.$
\begin{cor} \label{remg} The algebra $\mathfrak{gl}_{k,m}$ generated by indeterminates $h_z,$ $z\in X_k^+,$ $L^{(1)}_{x,y},$ $x,y \in [n]$ and relations,
\begin{eqnarray}
&& \big [ h_x,\ h_y \big ] =0, ~~x,y \in X_k^+ \nonumber\\ 
&& \big [L^{(1)}_{x,y},\ h_z \big ] = 2L_{x,y}^{(1)}h_x \delta_{z,x} - 2 h_y L_{x,y}^{(1)}\delta_{y,z}, ~~~{z\in X_k^+} \nonumber\\
&& \big [ L_{x,y}^{(1)},\ L_{z,w}^{(1)}\big ]-  
2L_{x,y}^{(1)}L^{(1)}_{x,w}\delta_{x,z}\big|_{x \in X_k^+} +2 L_{z,y}^{(1)} L^{(1)}_{x,y} \delta_{y,w} \big |_{y\in X_k^+} = L_{z,y}^{(1)}h_x\delta_{x,w}- h_yL_{x,w}^{(1)} \delta_{y,z} \nonumber \label{fund2b}
\end{eqnarray}
is a subalgebra of the Yangian $Y(\mathfrak{gl}_{k,m}).$\\ 
Note in particular $(L^{(1)}_{x,y})^2 =0,$ if $x \in [k]$ and $y \in X_k^+,$ or $x\in X_k^+$ and $y\in[k].$
\end{cor}
The nilpotency condition \((L^{(1)}_{x,y})^2 = 0\) serves as a mathematical hallmark of supersymmetry. Remarkably, it allows \(\mathfrak{gl}_{k,m}\) to function as a superalgebra despite being defined here in an entirely different framework that bypasses \(\mathbb{Z}_{2}\)-gradation.
\begin{rem} \label{gl}

$ $

\begin{enumerate}

\item There is a map $Y(\mathfrak{gl}_{k,m})\to\mathfrak{gl}_{k,m},$ such that $L^{(p)}_{x,y} \mapsto 0,$ for $p\geq2,$ i.e. $L^{(p)}_{x,y}(\lambda) \mapsto \delta_{x,y} h_x + \lambda^{-1} L^{(1)}_{x,y}$ and $L(\lambda) = L^{(0)} + \lambda^{-1} L^{(1)}.$

\item  Let $L^{(1)}_{x,x+1} =: e_x,$ $L^{(1)}_{x+1,x} =: f_x,$ $L^{(1)}_{x,x} =: \epsilon_x.$
Then every $L^{(1)}_{x,y},$ $|x-y|>1$ is derived by iteration via $\big [L^{(1)}_{x,y},\ L^{(1)}_{y,z}] = -h_yL^{(1)}_{x,z}$ for $x<y<z$ (or $x>y>z$). 
\end{enumerate}
\end{rem}

\begin{cor} \label{serre}
The following Chevalley-Serre type relations are satisfied by the generators $\big\{e_x,\ f_x\big\},$ $x\in [n-1]$ and $ \big \{\epsilon_x,\ h_x\big\},$ $x\in [n]$ ($\epsilon_x ,e_x,f_x$ are defined in Remark \ref{gl}):
\begin{eqnarray}
&& \big [h_x,\ h_y \big ] =0, ~~x,y \in X_k^+, ~~ \big [ \epsilon_z,\ h_x\big ] =0,~~z\in [n],
\nonumber\\ && \big [e_x,\ h_z\big ] = \big [f_x,\ h_z\big ] =0, ~~ z \neq x, x+1,\nonumber\\
&& \big [\big [\xi_x,\ h_x\big ]\big ]  =\big [\big [\xi_x,\ h_{x+1}\big ]\big ]  =0, ~\xi_x \in \big \{ e_x,\ f_x,\big \} x, x+1 \in X_k^+,\nonumber\\
&& \big [\epsilon_y,\ e_x \big ] =\big [\epsilon_y,\ f_x \big ] =0, ~~y \neq x,x+1,\nonumber\\
&& \big [\epsilon_x,\ e_x \big ] = -e_x, ~~\big[\epsilon_x,\ f_x \big ] = f_x,  ~~x\in [k],\nonumber\\
&&\big [\epsilon_{x+1},\ e_x\big ] =e_x, ~\big [\epsilon_{x+1},\ f_x\big ] =-f_x, ~~x+1 \in [k],\nonumber\\
&& \big [\big [\epsilon_x,\ e_x \big ]\big ] = h_xe_x, ~~
\big [\big[\epsilon_x,\ f_x \big ] \big ] = -h_xf_x,  ~~x \in X_k^+,\nonumber\\
&& \big [\big [\epsilon_{x+1},\ e_x\big ]\big ] =-h_{x+1}e_x, ~~\big [\big [\epsilon_{x+1},\ f_x\big ]\big ] =h_{x+1}f_x, ~~x+1 \in X_k^+,\nonumber\\
&& \big [f_x,\ e_x \big ] =  \epsilon_x h_{x+1} - \epsilon_{x+1} h_x, ~~x \in [n-1], \nonumber \\ &&
\big[f_x,\ e_y \big ] =0, ~y\neq x \in [k]~\mbox{and }~ y \neq x,x\pm 1, ~\mbox{if}~ x, x+1 \in X_k^+,\nonumber\\
&& \big[ \big [f_x, e_{x+1} \big ] \big ] = 0, ~~x+1\in X_k^{+},~~~  \big[ \big [f_x, e_{x-1} \big ] \big ]  =0, ~~x \in X_k^+, \nonumber 
\end{eqnarray}
where $ [[\ ,\  ]]: \mathfrak{gl}_{k,m} \times \mathfrak{gl}_{k,m}\to \mathfrak{gl}_{k,m},$ such that $[[{\mathrm a},\ {\mathrm b} ]] = {\mathrm a}{\mathrm b} + {\mathrm b}{\mathrm a},$ ${\mathrm a}, {\mathrm b} \in \mathfrak{gl}_{k,m},$ and
recall $h_z \in \mathfrak{gl}_{k,m}$ are invertible elements, such that $h_z =1,$ if $z \in [k]$ and $h_z \neq 1$ if $z \in X_k^+.$ 

Also, $e_k^2 =f_k^2 =0,$  and for all $x \in [n],$ and  $\xi_k \in \big \{e_k,\ f_k \big\},$
\begin{equation}
\xi_{x\pm 1}\ \xi_x^2 + \xi_x^2\ \xi_{x\pm 1} - 2 \xi_x\ \xi_{x\pm1}\ \xi_x =0,~~x\neq k,\label{cubic}
\end{equation}
and
\begin{equation}
\big [\big [\  [\xi_{k+1}, \xi_{k} ],\ [ \xi_{k}, \xi_{k-1}]\  \big]\big ] =0.
\label{quar}
\end{equation}
\end{cor}
\begin{proof}
All the exchange quadratic relations among the generators follow immediately by Corollary \ref{remg} and Remark \ref{gl} part (2). The cubic relations (\ref{cubic}) in particular are derived by employing:\\
$\big [L^{(1)}_{x,y},\ L^{(1)}_{y,z}] = -h_yL^{(1)}_{x,z}$ for $x<y<z$ (or $x>y>z$)\\
$\big [L^{(1)} _{x,y},\ L^{(1)} _{x,z}\big ] =\big [L ^{(1)}_{y,x},\ L^{(1)}_{z,x}\big ]=0,$  if $x \in [k]$ and $\big [\big [L^{(1)}_{x,y},\ L^{(1)}_{x,z}\big ]\big ] =\big [ \big [L^{(1)} _{y,x},\ L^{(1)}_{z,x}\big ] \big ] =0,$ if $x \in X_k^+,$ for all $x \neq y \neq z \in [n]$.  The quartic relation (\ref{quar}) for $x =k$ emerges from $ \big[ L^{(1)}_{k+2, k},\ L^{(1)}_{k+1, k-1} \big ] =0 = \big [L^{(1)}_{k,k+2},\ L^{(1)}_{k-1, k+1}\big]$ and the second part of Remark \ref{gl}.
\end{proof}
\begin{exa} \label{gl11}
We consider the simple example of the algebra $\mathfrak{gl}_{1,1},$ generated by elements $h := h_2,\ e :=e_1,\ f :=f_1,$ $\epsilon_1,\ \epsilon_2$ and relations, 
\begin{eqnarray}
&& \big [ \epsilon_1,\ \epsilon_2\big ] = \big [h, \epsilon_j \big ] =0, ~~ j\in [2], ~~~~ \big [\epsilon_1,\ e \big ] = -e, ~~~~\big [\epsilon_1,\ f \big ] =f \nonumber\\
&& \big [\big [\epsilon_2,\ e \big ] \big ] =eh, ~~~~ \big [\big [ \epsilon_2, f   \big ] \big ]  = -hf, ~~~~ \big [\big[f,\ h \big]\big ] =0 , ~~~~ \big [\big[ e,\ h\big ]\big ] =0 \nonumber \\  && \big [f,\ e \big ] =  \epsilon_1h - \epsilon_2,~~~~e^2 = f^2 =0. \nonumber
\end{eqnarray}
\end{exa}
In the following remark we consider a special case of the algebra $\mathfrak{gl}_{k,m}$ and explore the connection and subtle structural differences to the standard \(\mathfrak{gl}(k\vert{}m)\) superalgebra \cite{Kac, Zhang}.
\begin{rem} \label{wow}
If $h_x^2 =1$ for all $x \in [n],$ and we set $\hat e_x := h_x e_x,$ $\hat f_x : = h_{x+1}f_x,$  $x\in [n-1]$ ($h_x =1$, if $x \in [k]$) and $\hat \epsilon_{x} = \epsilon_x$ if $x \in [k],$ $\hat \epsilon_x = - h_x \epsilon_x$ if $x \in X_k^+,$ then the following algebraic relations are deduced for the elements $h_x, \hat \epsilon_x,$ $x\in [n],$ and  $\hat f_x, \hat e_x,$ $x\in [n-1]$ (we only write the non-zero commutation relations):
\begin{eqnarray}
&& \big [\big [\hat \xi_x,\ h_x\big ]\big ]  =\big [\big [\hat \xi_x,\ h_{x+1}\big ]\big ]  =0, ~\xi_x \in \big \{ \hat e_x,\ \hat f_x,\big \}, ~~ x,\ x+1 \in X_k^+ \label{a1} \nonumber
\end{eqnarray}
\begin{eqnarray}
&& \big [\hat \epsilon_x,\ \hat e_x \big ] = -\hat e_x, ~~\big[\hat \epsilon_x,\ \hat f_x \big ] = \hat f_x,  ~~x\in [n-1],\label{main1}\\
&&\big [\hat \epsilon_{x+1},\ \hat e_x\big ] =\hat e_x, ~\big [\hat \epsilon_{x+1},\ \hat f_x\big ] =-\hat f_x, ~~x \in [n-1], \label{main3}\\
&& \big [\hat f_{x},\ \hat e_x\big ] = \hat \epsilon_x - \hat \epsilon_{x+1}, ~~ k\neq x\in [n], ~~~\big [\big [\hat f_{k},\ \hat e_k\big ] \big ]= \hat \epsilon_k + \hat \epsilon_{k+1}. \label{main4}
\end{eqnarray}
and
\begin{eqnarray}
&& \big[ \big [\hat f_x, \hat e_{x+1} \big ] \big ] = 0, ~~x+1\in X_k^{+},~~~  \big[ \big [\hat f_x, \hat e_{x-1} \big ] \big ]  =0, ~~x \in X_k^+, \label{extra}
\end{eqnarray}
Also, $\hat e_k^2 =\hat f_k^2 =0,$ and for all $k\neq x \in [n],$ $\hat \xi_x \in \big \{\hat e_x,\ \hat f_x\big \}$ the cubic relations become,
\begin{eqnarray}
&& \hat \xi_{x\pm 1}\ \hat \xi_x^2 + \hat \xi_x^2\ \hat \xi_{x\pm 1} - 2 c_x \hat \xi_x\ \hat \xi_{x\pm1}\ \hat \xi_x =0, ~~~~~
c_x =\Bigg  \{ \begin{matrix} & ~~~~~1, \quad \mbox{if}~~~x \in [k-1],  \\ 
 & -1, \quad \mbox{if}~~~x\in X_k^+. \end{matrix} \nonumber
\label{cubic2}\end{eqnarray}
We observe that the main quadratic relations among the elements 
\(\hat {\epsilon}_x\), 
\(\hat {e}_{x}\), and \(\hat {f}_{x}\) (\ref{main1})-(\ref{main4})
are identical to those of the general linear Lie superalgebra \(\mathfrak{gl}(k\vert{}m)\), however extra quadratic anti-commutation relations exist (\ref{extra}). The cubic relations are also slightly modified relative to \(\mathfrak{gl}(k\vert{}m)\); specifically, \(c_x = -1\) if \(x \in X_k^+.\) Analogous modifications occur in the quartic relations (\ref{quar}), which we omit here for brevity.

Note that for the algebra $\mathfrak{gl}_{1,1}$, the elements \(\hat{e}_{1}\), $\hat{f}_{1}$, \(\hat{\epsilon }_{1}\), 
and \(\hat{\epsilon }_{2}\) correspond precisely to the Chevalley–Serre generators of the standard algebra \(\mathfrak{gl}(1|1)\). We provide further discussion on this connection after establishing the Hopf algebra structure of \(Y(\mathfrak{gl}_{k,m})\) (see also Remark \ref{repr}, part (3)).
\end{rem}
\begin{rem} \label{remc}
Any quantum algebra ${\mathfrak A}$ is also equipped with a coproduct $\Delta^{(N)}: {\mathfrak A} \to {\mathfrak A}^{ \otimes N} $ \cite{FRT, Drinfeld}, defined as follows.
Define\footnote{We usually do not present the indices $2,3,\ldots, N+1$ in $T_{1;23\ldots N+1}$ and simply write $T_{1}.$},
\begin{equation}
{\mathrm T}_{1;23\ldots N+1}(\lambda): = (\id \otimes \Delta^{(N)}) L(\lambda) = L_{1N+1}(\lambda) \ldots L_{13}(\lambda) L_{12}(\lambda).\  
\end{equation}
${\mathrm T}$ satisfies relation (\ref{RTT}) (i.e. $\Delta^{(N)}$ is an algebra homomorphism; note $\Delta^{(2)} =: \Delta$) and is expressed as $(\id \otimes \Delta^{(N)}) L(\lambda) = \underset{x,y \in X}{\sum} e_{x,y} \otimes \Delta^{(N)}(L_{x,y}(\lambda)).$ Coassociativity also holds, i.e.\\ $(\Delta \otimes \id)\Delta^{(N-1)} =(\id \otimes \Delta)\Delta^{(N-1)}= \Delta^{(N)}.$
\end{rem}
The quantum algebra \(Y(\mathfrak{gl}_{k,m})\) naturally possesses a Hopf algebra structure (see the Proposition below). This follows directly from Remark \ref{remc}, which defines the coproduct. Moreover, as already mentioned, the Yangian $Y(\mathfrak{gl}_{k,m})\otimes Y( \mathfrak{gl}_{k,m})$ possesses the standard tensor product algebra structure in contrast to the known Yangian of the general linear Lie super-algebra \cite{Kac, Molev, Nazarov, Zhang}. 
\begin{pro} \label{hopf}
Let $Y(\mathfrak{gl}_{k,m})$ be the  unital, associative algebra derived in Theorem \ref{alg1}. Then $(Y(\mathfrak{gl}_{k,m}), \Delta, \epsilon, s)$ is a Hopf algebra with the following maps:
\begin{enumerate}
\item a coproduct $\Delta: Y(\mathfrak{gl}_{k,m}) \to Y(\mathfrak{gl}_{k,m}) \otimes Y(\mathfrak{gl}_{k,m}),$ such that\\ $\Delta(h_x) = h_x \otimes h_x,$ for all $x \in X_k^+$ and
$\Delta(L^{(p)}_{x,y}) = \underset{z\in [n]}{\sum}\  \underset{\underset{p_1, p_2 \in {\mathbb N}}{p_1+p_2=p}}{\sum} L_{z,y}^{(p_1)} \otimes L_{x,z}^{(p_2)},$  for all $\ x,y\in [n].$

\item a counit $\epsilon: Y(\mathfrak{gl}_{k,m})\to {\mathbb C},$ such that $\epsilon(h_x) =1,$ for all $x\in X_k^+$ and $\epsilon(L^{(p)}_{x,y}) =0,$ for all $x,y \in [n],$ $p \in {\mathbb N}^*.$

\item an antipode $s: Y(\mathfrak{gl}_{k,m}) \to Y(\mathfrak{gl}_{k,m}),$ such that $s(h_x)= h_x^{-1},$ for all $x \in X_k^+$ and\\ 
\begin{equation}\underset{z\in [n]}{\sum}\ \underset{\underset{p_1, p_2 \in {\mathbb N}}{p_1+p_2=p}}{\sum} s(L_{z,y}^{(p_1)}) L_{x,z}^{(p_2)} =\underset{z\in [n]}{\sum}\ \underset{\underset{p_1, p_2 \in {\mathbb N}}{p_1+p_2=p}}{\sum} L_{z,y}^{(p_1)} s(L^{(p_2)}_{x,z}) =0.
\label{anti}\end{equation}
\end{enumerate}
\end{pro}
\begin{proof}

$ $

\begin{enumerate}
\item The coproduct is given in Remark \ref{remc}: $(\id \otimes \Delta)L(\lambda) =L_{13}(\lambda) L_{12}(\lambda),$ and recall $L(\lambda) = \underset{p\in {\mathbb N}}{\sum} \lambda^{-p}L^{(p)},$ $L^{(p)} = e_{x,y} \otimes L^{(p)}_{x,y},$ $p>0$ and $L_{x,y}^{(0)} = \delta_{x,y} h_x$
then 
\begin{equation}
(\id \otimes \Delta)L^{(p)} = \underset{\underset{p_1+p_2 = p}{p_1, p_2 \in {\mathbb N}}}{\sum}L^{(p_1)}_{13} L^{(p_2)}_{12}\ \Rightarrow\  \Delta(L^{(p)}_{x,y}) =\underset{\underset{p_1+p_2 = p}{p_1, p_2 \in {\mathbb N}}}{\sum}\ \sum_{z\in [n]} L_{z,y}^{(p_1)} \otimes L_{x,z}^{(p_2)}.\nonumber
\end{equation}
For instance, the coproducts of the first few order terms read as follows for all $x,y \in [n],$\\
$\Delta(h_x) = h_x \otimes h_x, ~~x\in X_k^{+},$\\
$\Delta(L^{(1)}_{x,y}) = L^{(1)}_{x,y} \otimes h_x + h_y\otimes L^{(1)}_{x,y},$\\
$\Delta(L^{(2)}_{x,y}) = L^{(2)}_{x,y}  \otimes h_x + h_y\otimes L^{(2)}_{x,y} +\underset{z\in [n]}{\sum} L^{(1)}_{z,y}\otimes L^{(1)}_{x,z},$\\
and so on.

\item  The counit is uniquely identified via, $(\epsilon \otimes \id )\Delta(a) = (\id \otimes \epsilon)\Delta(a) = a,$ for all $a \in Y(\mathfrak{gl}_{k,m}),$
which lead to $\epsilon(h_x) =1,$ $x\in X_k^+$ (group elements) and $\epsilon(L^{(p)}_{x,y}) =0,$ $p \in {\mathbb N}^*.$

\item  The antipode is uniquely identified by, $m(s\otimes \id)\Delta(a)  = m(\id \otimes s)\Delta(a) = \epsilon(a) 1_{Y}$ ($m(a\otimes b) = ab,$ $a,b\in Y(\mathfrak{gl}_{k,m}),$ $1_Y$ is the unit element in the Yangian), which lead to (\ref{anti}).

The antipode for all the elements of the algebra is identified by (\ref{anti}) by iteration. For instance, for all $x, y \in [n],$\\
$s(h_x)=h_{x^{}}^{-1},$\\
$s(L^{(1)}_{x,y}) = - h_y^{-1}L^{(1)}_{x,y}h_x^{-1} $\\
$s(L^{(2)}_{x,y}) = - h_y^{-1}L^{(2)}_{x,y}h_x^{-1} +\underset{{z\in[n]}}{\sum}h_y^{-1}L^{(1)}_{z,y} h_z^{-1} L^{(1)}_{x,z}h_x^{-1},$\\
and so on for higher order terms.\qedhere
\end{enumerate}
\end{proof}
As emphasized in the Introduction and prior to Proposition \ref{hopf}, the Yangian \(Y(\mathfrak{gl}_{k,m})\) is viewed as a Hopf algebra equipped with the standard, ungraded tensor product structure. This ungraded structure requires deforming the coproduct as demonstrated in Proposition \ref{hopf}, alongside the inclusion of a family of commutative group elements \(h_{x}\), \(x\in [n]\) within the algebra. Both the coproduct deformation and the inclusion of the group elements are naturally induced via FRT formulation. Consequently, this framework absorbs the super-type symmetries entirely into the deformed coproduct while preserving a completely ungraded tensor product algebra structure (see also relevant Remark \ref{wow}).
\begin{rem} \label{repr}

$ $

\begin{enumerate}

\item The map
$\pi: \mathfrak{gl}_{k,m} \to \EEnd({\mathbb C}^n),$ such that for all $x,y \in [n],$
\begin{equation}
L^{(1)}_{x,y}\mapsto e_{y,x},~~~h_x \mapsto \id_n, ~~ \mbox{if} ~~~x\in [k], ~~~ h_x \mapsto d_x := \id_n -2 e_{x,x}, ~~\mbox{if} ~~x\in X_k^+ \label{fundrep}
\end{equation}
is an algebra representation (the fundamental representation). 

 \item Consider $L(\lambda) = L^{(0)} + \lambda^{-1}L^{(1)},$ where $L^{(0)} = \underset{x \in [n]}{\sum}e_{x,x}\otimes h_x,$ 
 $L^{(1)} = \underset{x,y \in [n]}{\sum} e_{x,y} \otimes L^{(1)}_{x,y}$ and $h_x, L^{(1)}_{x,y} \in \mathfrak{gl}_{k,m}$, 
 recall also the fundamental representation given in (\ref{fundrep}), then $~(\id \otimes \pi )L^{(0)} = {\mathrm D}$ and $~(\id \otimes \pi) L^{(1)} = {\mathrm P}, $
where recall  ${\mathrm D}= \underset{x\in [n]}{\sum} e_{x,x } \otimes d_x =  \underset{x\in [n]}{\sum} d_{x} \otimes e_{x,x}$ and ${\mathrm P}= \underset{x,y \in [n]}{\sum} e_{x,y} \otimes e_{y,x},$
i.e. $(\id \otimes \pi )L(\lambda) = R(\lambda).$

\item From Remark \ref{wow} ($h_x^2 =1$) and Proposition \ref{hopf} the coproducts for $\hat e_x, \hat f_x,$ $x\in [n-1]$ and $ \hat \epsilon_x,$ $x \in [n]$ are derived (recall $h_x =1$ if $x \in [k],$ and $\Delta(h_x) = h_x \otimes h_x,$ $x \in X_k^+$):
\begin{eqnarray}
&& \Delta(\hat \xi_x) = \hat \xi_x \otimes 1 + h_{x} h_{x+1} \otimes \hat  \xi_x, ~~~\hat \xi_x \in \big\{\hat e_x, \hat f_x\big \}\nonumber\\
&&
\Delta(\hat \epsilon_x) = \hat \epsilon_x \otimes 1 + 1 \otimes \hat \epsilon_x. \nonumber
\end{eqnarray}
\end{enumerate}
\end{rem}
In the following proposition we derive central elements of $\mathfrak{gl}_{k,m}.$ We first introduce the notion of a trace 
for elements $A = \underset{x,y \in [n] }{\sum} e_{x,y} \otimes A_{x,y}\in \EEnd({\mathbb C}^n) \otimes {\mathfrak A},$ where ${\mathfrak A}$ is some quantum algebra as defined in Definition \ref{frt1}. Then  we define the trace of $A$, $tr(A) := \underset{x\in [n]}{\sum} A_{x,x} \in {\mathfrak A}.$

\begin{pro} \label{symm} Let $L = L^{(0)} + \lambda^{-1} L^{(1)}\in \EEnd({\mathbb C}^n) \otimes \mathfrak{gl}_{k,m},$ where $L^{(0)} = \underset{x \in [n]}{\sum} e_{x,x} \otimes h_x$ and $L^{(1)} = \underset{x,y \in [n]}{\sum} e_{x,y} \otimes L^{(1)}_{x,y}$ such that for all $x,y \in [n],$ $h_x, L^{(1)}_{x,y} \in \mathfrak{gl}_{k,m}$ ($n = k+m$). Let also $t(\lambda) : = L\big (\lambda) L^{-1}(-\lambda) \in \EEnd({\mathbb C}^n) \otimes \mathfrak{gl}_{k,m}$ (see also \cite{Sklyanin}) and $\tau(\lambda) := tr \big ( d\ t(\lambda) \big ) \in \mathfrak{gl}_{k,m},$ where $d = \underset{x\in [n]}{\sum} \theta_xe_{x,x}$ and $\theta_x =\Big  \{ \begin{matrix} & 1, ~~if~~x \in [k]\\ 
& -1, ~~if ~~x \in X_k^+ \end{matrix}.$ Then,
 $\big [\tau(\lambda),\  \mathfrak{g} \big ] =0,$ $\mathfrak{g}\in \mathfrak{gl}_{k,m}.$
\end{pro}
\begin{proof} We first recall that $L(\lambda) = L^{(0)} + \lambda^{-1}L^{(1)}$ satisfies equation (\ref{RTT}) with $R(\lambda) = r + \lambda^{-1} {\mathrm P},$ where $r = {\mathrm D}$ is given in Proposition \ref{pro11} and ${\mathrm P}$ is the permutation map. Then from (\ref{RTT}) (see also Theorem \ref{alg1} and Remark \ref{repr}) we obtain
$
~(\pi \otimes \id)\Delta^{(op)}(\mathfrak{g}) L(\lambda) = L(\lambda)(\pi \otimes \id)\Delta(\mathfrak{g}), ~~~\mathfrak{g} \in \mathfrak{gl}_{k,m},$\\
where recall $\Delta ^{(op)} = \sigma \circ \Delta,$ $\sigma$ is the algebra flip map, $\sigma: Y(\mathfrak{gl}_{k,m}) \otimes Y(\mathfrak{gl}_{k,m}) \to Y(\mathfrak{gl}_{k,m}) \otimes Y(\mathfrak{gl}_{k,m}),$ $a \otimes b \mapsto b \otimes a$. We then deduce,
$~(\pi \otimes \id)\Delta(\mathfrak{g}) L^{-1}(-\lambda) = L^{-1}(-\lambda)(\pi \otimes \id)\Delta^{(op)}(\mathfrak{g}), ~~~\mathfrak{g} \in \mathfrak{gl}_{k,m}, $
and conclude that
\begin{equation}
(\pi \otimes \id)\Delta^{(op)}(\mathfrak{g}) t(\lambda) = t(\lambda)(\pi \otimes \id)\Delta^{(op)}(\mathfrak{g}). \label{dbasic} \end{equation}

We first examine the case $\mathfrak{g} \in \big \{h_z|\ z \in X_k^+\big \};$  from (\ref{dbasic}), 
Remark \ref{repr} and expressing $t(\lambda) = \underset{x, y\in [n]}{\sum e_{x,y} \otimes t_{x,y}(\lambda)}$ we deduce
\begin{eqnarray}
&& (d_z \otimes h_z) t(\lambda) = t(\lambda)( d_z \otimes h_z )\  \Rightarrow\ \big [ \big [h_z, t_{x,z}(\lambda) \big ]\big ] =\big [ \big [h_z, t_{z,x}(\lambda) \big ]\big ] = 0, \nonumber\\ && \big[ h_z,\ t_{z,z}(\lambda)\big ] = \big[ h_z,\ t_{x,y}(\lambda)\big ] =0 ~~\mbox{if} ~~z\neq x,y \in [n]. \label{equ1}
\end{eqnarray}
By equations (\ref{equ1}) we conclude, $ \big [ h_{z},\ \tau(\lambda) \big] =0,$ for all $z \in [n].$

If $\mathfrak{g} \in \big \{L^{(1)}_{x,y}|\ x,y \in [n]\big \},$  we deduce from (\ref{dbasic}) for all $x,y,z,w \in [n],$
\begin{eqnarray}
    && \big [L^{(1)}_{x,y},\ t_{z,z} \big ] =0, ~~~~z \neq x,y\nonumber\\
    && \big [L^{(1)}_{x,y},\ t_{y,y} \big ]= - h_yt_{x,y} \nonumber\\
    && \big [L^{(1)}_{x,y},\ t_{x,x} \big ] =  \theta_xt_{x,y}h_y. \label{equ2}
\end{eqnarray}
Then from (\ref{equ1}) and (\ref{equ2}) we conclude $\big [ L^{(1)}_{x,y},\ \tau(\lambda) \big] =0,$ for all $x,y \in [n].$
\end{proof}

\begin{exa}

The first two non-trivial elements of the expansion $\tau(\lambda) = \underset{n\geq 0}{\sum}\lambda^{-n}\tau^{(n)}$ are
\[ \tau^{(1)} = 2\sum_{x\in [n]}\theta_xL^{(1)}_{x,x}h_x^{-1} , ~~~~\tau^{(2)} = 2\sum_{x,y \in [n]} \theta_xL_{x,y}^{(1)} h_y^{-1} L_{y,x}^{(1)}h_x^{-1}.\]
The elements $\tau^{(1)},$  $\tau^{(2)}$ are the linear and quadratic Casimir elements of $\mathfrak{gl}_{k,m}$ respectively. By setting $\hat L_{x,y}: = h_x^{-1} L^{(1)}_{x,y}$ (recall Remark \ref{wow}), the Casimir elements become, $\tau^{(1)} = 2\underset{x \in [n]}{\sum} \theta_x \hat L_{x,x}$ and $ ~\tau^{(2)} = 2\underset{x,y \in [n]}{\sum} \theta_x \hat L_{x,y} \hat  L_{y,x}.$
\end{exa}

\section{Mutual centralizers and tensor representations}
\noindent The primary objective of this section is to investigate the symmetry of involutive tensor representations of the braid group, and to derive specific finite-dimensional irreducible representations of $\mathfrak{gl}_{k,m}$
with their associated bases.
\begin{pro} \label{symm0}
Let $\rho: B_N \to \EEnd(({\mathbb C}^n)^{\otimes N}),$ $\sigma_j \mapsto \check r_j,$ where
$\check r_j := \id^{\otimes (j-1)}  \otimes  \check r \otimes \id^{\otimes (N-j-1)}$ and
$ \check r \in \EEnd( {\mathbb C}^n\otimes{\mathbb C}^n)$  is the involutive solution of the braid equation given in Proposition \ref{pro11}. Let also 
$\pi: \mathfrak{gl}_{k,m} \to \EEnd({\mathbb C}^n)$ 
be the fundamental representation of $\mathfrak{gl}_{k,m}$ (\ref{fundrep}).
Then 
\begin{equation}
\big [\check r_j,\ \pi^{\otimes N}\Delta^{(N)}(\mathfrak{g}) \big ] = 0, 
~~\forall\ \mathfrak{g}  \in \mathfrak{gl}_{k,m}, ~~ j \in [N-1]
\end{equation}
\end{pro}
\begin{proof}
We first recall the following relations derived in Theorem \ref{alg1} (see also footnote 1)
\begin{equation}
\check r_{12} L^{(0)}_{1} L^{(0)}_{2} = L^{(0)}_{1} L_{2}^{(0)} \check r_{12}, ~~~~ \check r_{12} L^{(0)}_{1} L^{(1)}_{2} = L^{(1)}_{1} L_2^{(0)} \check r_{12}, ~~~~ \check r_{12} L^{(1)}_{1} L^{(0)}_{2} = L^{(0)}_{1} L_2^{(1)} \check r_{12}\label{basiceq}
\end{equation}
where $L^{(0)} = \underset{x \in[n]}{\sum}e_{x,x} \otimes h_x,$ $L^{(1)} = \underset{x,y \in [n]}{\sum} e_{x,y} \otimes L^{(1)}_{x,y},$
and from Remark \ref{repr}, equations hold for $L^{(0)}  \mapsto {\mathrm D}$ and $L^{(1)} \mapsto {\mathrm P}.$ Then from the first  equation of (\ref{basiceq}):
\begin{eqnarray}
    && \check r  (\pi \otimes \pi)\Delta(h_x) =  (\pi \otimes \pi)\Delta(h_x)\ \check r \label{eq21}
    \end{eqnarray}
and from the last two equations of (\ref{basiceq}):
$
(\pi \otimes \pi)\Delta^{(op)}(L_{x,y}^{(1)})\check r = \check r(\pi \otimes \pi)\Delta^{(op)}(L_{x,y}^{(1)}).$
 Recall also that $\Delta(L^{(1)}_{x,y}) = L_{x,y} \otimes h_x + h_y \otimes L_{x,y},$ then $\Delta^{(op)}(L^{(1)}_{x,y}) = L_{x,y} \otimes h_y + h_x \otimes L_{x,y}.$ 
We observe that $\check r_{12} = \check r_{21},$ which leads  to 
\begin{equation}
    (\pi \otimes \pi)\Delta(L_{x,y}^{(1)})\check r = \check r(\pi \otimes \pi)\Delta(L_{x,y}^{(1)}).\label{eq23}
\end{equation}

We conclude from (\ref{eq21}), (\ref{eq23}), $\big [\check r,\ (\pi \otimes \pi)\Delta(\mathfrak{g}) \big ]=0$ for all $\mathfrak{g} \in \mathfrak{gl}_{k,m},$ then recalling also the $N$-coproducts and the definition of $\check r_j$ (\ref{tensor}) we deduce that $\big [ \check r_j,\ \pi^{\otimes N} \Delta(\mathfrak{g})\big ]=0$ for all $j\in [N-1]$ and $\mathfrak{g} \in \mathfrak{gl}_{k,m}$
\end{proof}

Henceforth, we drop the superscript in $L^{(1)}_{x,y} \in \mathfrak{gl}_{k,m}$ and write $L_{x,y}$ (see Corollary \ref{remg}).
\begin{rem} \label{nota1} ({\bf Notation})
We introduce at this point some notation to be used in the following propositions. 
\begin{enumerate}

\item  We introduce the shorthand notation:
\begin{equation}
 \qquad \quad  {\cal E}_x:= \pi^{\otimes N}\Delta^{(N)}(\epsilon_x), ~~{\mathfrak h}_x : = \pi^{\otimes N}\Delta^{(N)}(h_x), ~~{\mathfrak e}_x:=  \pi^{\otimes N}\Delta^{(N)}(e_x), ~~{\mathfrak f}_x : = \pi^{\otimes N} \Delta^{(N)}(f_x), \nonumber
\end{equation}
and ${\mathfrak t}_{x,y} : =  \pi^{\otimes N}\Delta^{(N)}(L_{x,y}),$ $x, y\in [n].$
Recall $\epsilon_x,$ $x\in [n],$ $\ e_x,\ f_x,$ $x \in [n-1]$ are defined in Remark \ref{gl} and the map $\pi: \mathfrak{gl}_{k,m} \to \EEnd({\mathbb C}^n)$ is defined in Remark \ref{repr}.

\item  Let $\mathfrak{u}^+ : = \underbrace{e_1 \otimes e_1 \otimes \ldots  \otimes e_1}_N$ and $\mathfrak{u}^- : =\underbrace{e_n\otimes e_n \otimes\ldots \otimes e_n}_N.$ Notice, ${\mathfrak f}_x \mathfrak{u}^+ =0,$ ${\mathfrak e}_x \mathfrak{u}^- = 0,$ for all $x \in [n-1].$ 
Also, utilizing the notation of part (1) above we introduce the vectors $\mathfrak{u}^{\pm}_{m_1, m_2, \ldots, m_n} \in ({\mathbb C}^n)^{\otimes N}:$ 
\begin{eqnarray}
&& \mathfrak{u}^{+}_{m_1, m_2, \ldots, m_n} := {\mathfrak t}^{m_n}_{1,n} {\mathfrak t}^{m_{n-1}}_{1,n-1}\ldots  {\mathfrak t}_{1,3}^{m_3} {\mathfrak t}_{1,2}^{m_2} \ \mathfrak{u}^+, ~~~~ \mathfrak{u}^{-}_{m_1, m_2, \ldots, m_n} :=   {\mathfrak t}_{n, n-1}^{m_{n-1}}{\mathfrak t}^{m_{n-2}}_{n,n-2} \ldots {\mathfrak t}^{m_2}_{n,2} {\mathfrak t}^{m_1}_{n,1}\ \mathfrak{u}^-, \label{vectors} \end{eqnarray}
$ \underset{x \in [n]}{\sum} m_x =N.$ We recall from Corollary (\ref{remg}) that $\big [{\mathfrak t}_{1,x},\ {\mathfrak t}_{1,y}\big] =0$ and $ \big [ \big [{\mathfrak t}_{n,x}, {\mathfrak t}_{n,y} \big ] \big] =0,$  for all $x,y\in [n].$ Henceforth, whenever we write $\underset{x}{\prod} \mathfrak t^{m_x}_{p,x},$ $p \in \big \{1, n\big \}$ we mean the ordered products of (\ref{vectors}). Moreover, due to $L^2_{x,y}=0$ for $x\in [k]$ and $y \in X_k^+$ or $x \in X_k^+$ and $y \in [k],$ we deduce that 
\begin{enumerate}
\item in $\mathfrak{u}^+_{m_1, m_2, \ldots, m_n},$ $m_x \in \big \{0,1, \ldots, N\big \}$ if $x \in [k]$ and $m_x \in \big \{0, 1\big \} $ if $x \in X_k^+.$
\item in $\mathfrak{u}^-_{m_1, m_2, \ldots, m_n},$ $m_x \in\big \{0,1, \ldots, N\big \}$ if $x \in X_k^+$ and $m_x \in\big \{0, 1\big \}$ if $x \in [k].$\end{enumerate}

\end{enumerate}
\end{rem}


\begin{thm}  \label{basic12} Let $L_{x,y} \in \mathfrak{gl}_{k,m}$ (see Corollary \ref{remg}) and  
$ \mathfrak{u}^{\pm}_{m_1,\ldots, m_n} \in ({\mathbb C}^n)^{\otimes N}$ be defined in Remark (\ref{nota1}).
Consider also the sum of all the generators of the braid group $B_N$, ${\cal H} : = \underset{j \in [N-1]}{\sum} \check r_j \in \EEnd(({\mathbb C}^n)^{\otimes N}),$
where $\check r$ is given in Proposition \ref{pro11}. Then:
\begin{enumerate}
\item  ${\cal H} \mathfrak{u}^{\pm}_{m_1, \ldots, m_n} = \pm (N-1) \mathfrak{u}^{\pm}_{m_1, \ldots, m_n}. $

\item  The action of $\mathfrak{gl}_{k,m}$ on $\mathfrak{u}^{\pm}_{m_1, m_2, \ldots, m_n},$  $\underset{y\in [n]}{\sum} m_y =N$ is given  by,
\begin{eqnarray}
&&{\mathfrak h}_x\ \mathfrak{u}^{\pm}_{m_1, \ldots, m_x, m_{x+1}, \ldots} = (-1)^{m_x}\ \mathfrak{u}^{\pm}_{m_1,\ldots, m_x, m_{x+1}, \ldots}, ~~x \in X_k^+
\nonumber\\
&& {\cal E}_x\ \mathfrak{u}^{\pm}_{m_1,\ldots, m_x, m_{x+1},\ldots} =  m_x\ \mathfrak{u}^{\pm}_{m_1, \ldots, m_x, m_{x+1}, \ldots}, ~~x \in [k] \nonumber\\
&& {\cal E}_x\ \mathfrak{u}^{\pm}_{m_1,\ldots, m_x, m_{x+1},\ldots} =  m_x\ (-1)^{m_x -1} \mathfrak{u}^{\pm}_{m_1, \ldots, m_x, m_{x+1}, \ldots}, ~~~x\in X_k^+ \nonumber\\
&& {\mathfrak e}_x\  \mathfrak{u}^{\pm}_{m_1, \ldots, m_x, m_{x+1}, \ldots } = b^{\pm}_x\ \mathfrak{u}^{\pm}_{ m_1,\ldots, m_x-1, m_{x+1}+1, \ldots },~~x\in[n-1] \nonumber\\
&& {\mathfrak f}_x\ \mathfrak{u}^{\pm}_{m_1, \ldots, m_x, m_{x+1}, \ldots }   = c^{\pm}_x\ \mathfrak{u}^{\pm}_{m_1, \ldots, m_x+1, m_{x+1}-1, \ldots },~~x\in [n-1], \nonumber \end{eqnarray}
where  
\begin{itemize}
\item for all $x \neq 1,$
$ b^+_x = m_x,$ $c^+_x =m_{x+1},$ if $ x+1\in [k];$\\
    $~b^+_x = (-1)^{m_{x+1}}m_x,$ $c^+_x = m_{x+1},$ if $ \ x+1 \in X_k^+,$ 
   \item  $b^+_1= 1,$ $c^+_1 = (m_1+1)m_2,$
 \end{itemize}
$ $
   \noindent and
   \begin{itemize}
   \item for all $x \neq n-1,$
 $ b^-_x = m_x,$ $c^-_x =m_{x+1},$ if $ x+1\in [k];$\\
  $~b^-_x = (-1)^{m_{x+1} +m_x-1}m_x,$ $c^-_x = m_{x+1},$ if $\  x+1 \in X_k^+,$ 
 \item $b^-_{n-1}= (-1)^{m_n +m_{n-1}-1}m_{n-1}(m_n+1),$ $c^-_{n-1} = 1.$   
  \end{itemize}
   \end{enumerate}
\end{thm}
\begin{proof}

$ $

\begin{enumerate}
\item We first recall the action of the elements $\check r_j,$ $j \in [N-1]$ on the standard canonical basis of $({\mathbb C}^{n})^{\otimes  N}$ given in (\ref{action1}). We then immediately deduce that ${\cal H}\ \mathfrak{u}^{\pm}=  \pm (N-1)\ \mathfrak{u}^{\pm}.$
 Due to Proposition \ref{symm0} $\big [ {\cal H}, \Delta^{(N)}(\mathfrak{g})\big ] =0,$ $\mathfrak{g} \in \mathfrak{gl}_{k,m,}$ hence it immediately follows that,\\
 ${\cal H}\ \mathfrak{u}^{\pm}_{m_1, m_2, \ldots, m_n} = \pm (N-1)\ \mathfrak{u}^{\pm}_{m_1, m_2, \ldots, m_n}.$

 \item By Corollary \ref{remg} and using the notation introduced in Remark \ref{nota1} for all, $x, y, z \in [n]$:
 \begin{eqnarray}
 && \big [\mathfrak{t}_{x,y},\ \mathfrak{t}_{z,x} \big ] = \mathfrak{t}_{z,y}\mathfrak{h}_x,~~~\big [\mathfrak{t}_{x,y},\ {\mathfrak h}_z\big ] =0, ~~z\neq x,y \nonumber \\ && \big [\big [\mathfrak{t}_{x,z},\ {\mathfrak h}_z\big ]\big] =\big [\big [\mathfrak{t}_{z,x},\ {\mathfrak h}_z\big ]\big] = 0, ~~z\in X_k^+,~~~z\neq x\nonumber\\ 
 && \big [ {\mathfrak t}_{x,y},\ {\mathfrak t}_{x,z} \big ] =0, ~~~~ \big [ \mathfrak{t}_{y,x},\ \mathfrak{t}_{z,x} \big ] =0,~~~ x\in [k], ~~~x\neq y,z \nonumber\\
 && \big [ \big [{\mathfrak t}_{x,y},\ {\mathfrak t}_{x,z} \big ] \big ]=0,~~~~\big [ \big [{\mathfrak t}_{y,x},\ {\mathfrak t}_{z,x} \big ] \big ]=0,~~~x \in X_k^+,~~~~x\neq y,z. \label{relations}
 \end{eqnarray}
 The action of ${\cal E}_x, {\mathfrak h}_x$ on $\mathfrak{u}^{\pm}_{m_1,m_2, \ldots, m_n}$ is straightforward, so we focus on the action of ${\mathfrak e}_x, {\mathfrak f}_x.$
 \begin{itemize}
     \item For $x+1 \in [k],$ $x \neq 1$
 \begin{eqnarray}
     {\mathfrak e}_x\ \mathfrak{u}^{+}_{m_1, \ldots, m_x, m_{x+1}, \ldots, m_n} &=& \prod_{x\neq y=2}^n \mathfrak{t}^{m_y}_{1,y}\ \mathfrak{e}_x {\mathfrak t}^{m_x}_{1,x}\ \mathfrak{u}^+ \nonumber\\ &  =&  \prod_{x\neq y=2}^n \mathfrak{t}^{m_y}_{1,y}\  \big ( {\mathfrak t}^{m_x}_{1,x} {\mathfrak e}_x + \sum_{p=0}^{m_x-1} \mathfrak{t}^p_{1,x} \mathfrak{t}_{1, x+1} \mathfrak{t}^{m_{x-1-p}}_{1,x}\big )\ \mathfrak{u}^+ \nonumber\\ 
     & =& m_x\ \mathfrak{u}^{+}_{m_1, \ldots, m_x-1, m_{x+1}+1, \ldots, m_n.} \nonumber \end{eqnarray}
where we have used $\big [{\mathrm a},\ {\mathrm b}^n \big ] = \underset{0\leq p \leq n-1}{\sum} {\mathrm b}^p \big [{\mathrm a},\ {\mathrm b}\big ] {\mathrm b}^{n-p-1},$  for
 ${\mathrm a}, {\mathrm b} \in Y(\mathfrak{gl}_{k,m}).$ Similarly, 
$~{\mathfrak f}_x\ \mathfrak{u}^{+}_{m_1, \ldots, m_x, m_{x+1}, \ldots, m_n} = m_{x+1}\  \mathfrak{u}^{+}_{m_1, \ldots, m_x+1, m_{x+1}-1, \ldots, m_n}.
$
\\
For $x =1 $ in particular and $k>1$: $ \mathfrak{e}_1\ \mathfrak{u}^+_{m_1, m_2, \ldots, m_n} = \mathfrak{u}^+_{m_1-1, m_2+1, 
\ldots, m_n}$\\ and via relations (\ref{relations}) it follows 
$ \mathfrak{f}_1\ \mathfrak{u}^+_{m_1, m_2, \ldots , m_n} = (m_1+1)m_2\ \mathfrak{u}^+_{m_1+1, m_2-1, \ldots, m_n}$

\item For $x+1 \in X_k^+,$ we similarly deduce from relations (\ref{relations})\\
${\mathfrak e}_x\ \mathfrak{u}^{+}_{m_1, \ldots, m_x, m_{x+1}, \ldots, m_n} = (-1)^{m_{x+1}}m_x\ \mathfrak{u}^{+}_{m_1, \ldots, m_x-1, m_{x+1}+1, \ldots, m_n}$ and\\ 
$~{\mathfrak f}_x\ \mathfrak{u}^{+}_{m_1, \ldots, m_x, m_{x+1}, \ldots, m_n} = m_{x+1}\ \mathfrak{u}^{+}_{m_1, \ldots, m_x+1, m_{x+1}-1, \ldots, m_n}.$ 
\\
For $x =1 $ in particular and $k=1$: $ \mathfrak{e}_1\ \mathfrak{u}^+_{m_1, m_2, \ldots , m_n} = \mathfrak{u}^+_{m_1-1, m_2+1, \ldots , m_n}$ and via\\ relations (\ref{relations}) it follows 
$ \mathfrak{f}_1\ \mathfrak{u}^+_{m_1, m_2, \ldots , m_n} = (-1)^{m_2-1} (m_1+1)m_2\ \mathfrak{u}^+_{m_1+1, m_2-1, \ldots , m_n},$ where in general $m_x+1$ is defined $\mod 2$ for all $x \in X_k^+$ in all the coefficients of the right-hand side of the equations of part 2 of Theorem \ref{basic12} and recall ${\mathfrak t}^2_{1,x} =0,$ for all $x \in X_k^+.$
We also use, $(-1)^{m-1} m = m,$ if $m \in \big \{0,\ 1\big \}.$
 \end{itemize}
Similarly, for the action of $\mathfrak{gl}_{k,m} $ on $ {\mathfrak u}^{-}_{m_1, m_2, \ldots, m_n};$ in this case $m_x+1$ is defined $\mod 2$ for all $x \in [k]$ in all the coefficients of the right-hand side of the equations of part 2 of Theorem \ref{basic12} and ${\mathfrak t}^2_{n,x} =0,$ for all $x \in [k].$
\qedhere
\end{enumerate}
 \end{proof}
Theorem \ref{basic12} (part 2) establishes that the sets $\big \{\mathfrak{u}^{\pm}_{m_1, m_2, \ldots, m_n}\big \}_{\sum m_j=N}$ form natural bases for distinct irreducible representations of $\mathfrak{gl}_{k,m}$. We refer to these as \textit{combinatorial bases}. A detailed investigation of irreducible representations of $\mathfrak{gl}_{1,1}$ is provided in the following section. Furthermore, we note that while any linear combination of the braid group generators can theoretically serve as a Hamiltonian, namely $  \mathcal{H} = \underset{j \in [N-1]}{\sum} c_j \hat r_j\in \mbox{End}(({\mathbb C}^n)^{\otimes N}),$ $~c_j \in \mathbb{C},$
this work focuses exclusively on the uniform sum of these generators. This choice recovers well-known Hamiltonians of quantum spin-chain-like systems subject to special boundary conditions (see, e.g., \cite{DoikouMartin, Doikou1, Sklyanin}; see also the Heisenberg XX model example in the final section). 

We consider next as an example relevant to Theorem \ref{basic12} the spectral decomposition of ${\cal H} = \underset{ j \in [N-1]}{\sum} \check r_j,$ when $N=2.$ In this case we analyze the spectral decomposition of $\check r,$ and produce the bases of the corresponding irreducible representations. As an illustrative example we graphically depict the action of $\mathfrak{gl}_{1,1}$ and $\mathfrak{gl}_{2,1}$ on these bases.
\begin{exa} \label{exa2} Let $\check r$ be the $\mathfrak{gl}_{k,m}$ invariant solution of the braid equation of Proposition \ref{pro11} and 
$\big \{e_x \big  \}_{x \in [n]}$ be the standard canonical basis of ${\mathbb C}^n.$ 

Then, $\check r$ has 2 eigenvalues: 
\begin{enumerate}
    \item $\lambda_1 =1$ with multiplicity $\frac{n(n-1)}{2} +k$ and corresponding (non-normalized) eigenvectors according to Theorem \ref{basic12} are (we simplify the notation compared to Theorem \ref{basic12}):\\
    $u^{+}_{1,1} := e_1 \otimes e_1,$ $~u^+_{j,j}: = {\mathfrak t}^2_{1,j}u^{+}_{1,1} = 2 e_j \otimes e_j,$ $j\in [k],$ \\
      $u^+_{1,j} := {\mathfrak t}_{1,j} u^{+}_{1,1} = e_1 \otimes e_j + e_j \otimes e_1,$ $ j \in \big \{2,3, \ldots, n\big \},$\\    $u^{+}_{i,j} := {\mathfrak t}_{1,j}{\mathfrak t}_{1,i} u^{+}_{1,1} = e_i \otimes e_j + e_j \otimes e_i,$ $i< j \in \big \{2,3, \ldots, n\big \}.$ 
    
    \item $\lambda_2 =-1$ with multiplicity $ \frac{n(n+1)}{2}-k$ and corresponding (non-normalized) eigenvectors (Theorem \ref{basic12})\\ 
    $u^{-}_{n,n} = e_n \otimes e_n,$ $\ u_{j,j}^{-} = {\mathfrak t}^2_{n,j}u^{-}_{n,n}= 2 e_j \otimes e_j, $ $\ j \in X_k^+,$\\ 
     $u^{-}_{n,j} :=  {\mathfrak t}_{n,j} u^{-}_{n,n}  = e_n \otimes e_j - e_j \otimes e_n,$ $\ j\in [n-1],$\\
    $u^{-}_{i,j} := {\mathfrak t}_{n,j}{\mathfrak t}_{n,i} u^{-}_{n,n} = e_j \otimes e_i - e_i \otimes e_j,$ $\ i<j \in [n-1].$

\end{enumerate}
That is, $V_n^{\otimes 2} = V_{\lambda_1} \oplus V_{\lambda_2}$ and $\dim V_n = n,$ $\ \dim V_{\lambda_1} =\frac{n(n-1)}{2} +k,$ $\ \dim V_{\lambda_2} =\frac{n(n+1)}{2} -k.$ 
Each eigenspace is invariant under the action of $\mathfrak{gl}_{k,m}$ (see Theorem \ref{basic12}) and the eigenvectors provide bases of irreducible representations of $\mathfrak{gl}_{k,m}$ of dimensions $d_1= \frac{n(n-1)}{2} +k,$ $d_2 = \frac{n(n+1)}{2} -k$. 

The action of $\mathfrak{gl}_{k,m}$  on the basis $\big \{u^{+}_{x,y}\big \},$ $x,y \in [n],$ (Example \ref{exa2}) is illustrated schematically for two special cases: $\mathfrak{gl}_{1,1}$ and $\mathfrak{gl}_{2,1}$. The diagrams below depict the transition from one element of the basis to another under the action of $\mathfrak{gl}_{k,m}$. However, the coefficients associated to each transition are omitted for brevity (see Theorem \ref{basic12}, part 2 about the transition coefficients).
\begin{enumerate}
\item $\mathfrak{gl}_{1,1},$ a two-dimensional basis: $\zeta \in \big \{{\cal E}_1, {\cal E}_2, \mathfrak{h}\big\}$
$ $

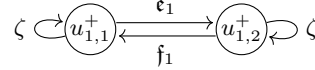
\begin{center}
\footnotesize
\begin{tikzpicture}[shorten >=1pt,node distance=2.0cm,on grid,auto] 
   \node[state,inner sep=0, minimum size=2.0em] (q_0)   {$u^{+}_{1,1}$}; 
   \node[state, inner sep=0, minimum size=2.0em] (q_1) [right =of q_0] {$u^{+}_{1,2}$};    
   \path[->] 
 (q_0) edge[transform canvas={yshift=2.5pt}, above]  node {$\mathfrak e_1$} (q_1)
 (q_0) edge[loop left] node{$\zeta$} (q_0) 
 (q_1) edge[loop right] node{$\zeta$} (q_1)
 (q_1) edge[transform canvas={yshift=-2.5pt}, below]  node {$\mathfrak f_1$} (q_0)
;
\end{tikzpicture}
\end{center}

\item $\mathfrak{gl}_{2,1},$ a five-dimensional basis: $\zeta \in \big \{{\cal E}_1, {\cal E}_2, {\cal E}_3, \mathfrak{h}_1, \mathfrak{h}_2\big\}$

$ $

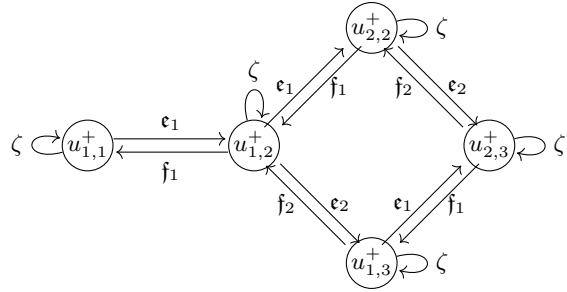
\begin{center}
\footnotesize
\begin{tikzpicture}[shorten >=1pt,node distance=2.2cm,on grid,auto] 
   \node[state, inner sep=0, minimum size=2.0em] (q_1) [below=of q_0] {$u^{+}_{2,2}$};    
   \node[state, inner sep=0, minimum size=2.0em] (q_2) [below left=of q_1] {$u^{+}_{1,2}$}; 
   \node[state, inner sep=0, minimum size=2.0em] (q) [left=of q_2] {$u^{+}_{1,1}$};
   \node[state, inner sep=0, minimum size=2.0em] (q_3) [below right=of q_1] {$u^{+}_{2,3}$}; 
    \node[state,inner sep=0, minimum size=2.0em](q_4) [below right=of q_2] {$u^{+}_{1,3}$};
   \path[->] 
 (q) edge[transform canvas={yshift=2.5pt}, above]  node {$\mathfrak e_1$} (q_2)
 (q_2) edge[transform canvas={yshift=-2.5pt}, below]  node {$\mathfrak f_1$} (q)
 (q_1) edge[transform canvas={xshift=3pt},right]  node {$\mathfrak f_1$} (q_2)
 (q_2) edge[transform canvas={xshift=-3pt},left]  node {$\mathfrak e_1$} (q_1)
 (q_1) edge[loop right] node{$\zeta$} (q_1) 
 (q) edge[loop left] node{$\zeta$} (q) 
 (q_2) edge[loop above] node{$\zeta$} (q_2) 
(q_3) edge[loop right] node{$\zeta$} (q_3)
(q_4) edge[loop right] node{$\zeta$} (q_4)  
(q_1) edge[transform canvas={xshift=3pt}, right] node {$\mathfrak e_2$} (q_3) 
 (q_3) edge[transform canvas={xshift=-3pt}, left] node {$\mathfrak f_2$} (q_1) 
(q_2) edge[transform canvas={xshift=3pt}, right]  node {$\mathfrak e_2$} (q_4)
(q_4) edge[transform canvas={xshift=-3pt}, left]  node {$\mathfrak f_2$} (q_2)
(q_3) edge[transform canvas={xshift=3pt}, right]  node {$\mathfrak f_1$} (q_4) 
(q_4) edge[transform canvas={xshift=-3pt}, left]  node {$\mathfrak e_1$} (q_3) 
;
\end{tikzpicture}
\end{center}
\end{enumerate}

The action of $\mathfrak{gl}_{k,m}$ on the basis $\big \{u^{-}_{x,y}\big \},$ $x,y \in [n],$ (Example \ref{exa2}) is also depicted schematically for the special cases $\mathfrak{gl}_{1,1}$ and $ \mathfrak{gl}_{2,1}:$

\begin{enumerate}
\item $\mathfrak{gl}_{1,1},$ a two-dimensional basis: $\zeta \in \big \{{\cal E}_1, {\cal E}_2, \mathfrak{h}\big\}$

$ $

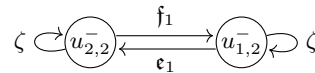
\begin{center}
\footnotesize
\begin{tikzpicture}[shorten >=1pt,node distance=2.0cm,on grid,auto] 
   \node[state,inner sep=0, minimum size=2.0em] (q_0)   {$u^{-}_{2,2}$}; 
   \node[state, inner sep=0, minimum size=2.0em] (q_1) [right =of q_0] {$u^{-}_{1,2}$};    
   \path[->] 
 (q_0) edge[transform canvas={yshift=2.5pt}, above]  node {$\mathfrak f_1$} (q_1)
 (q_1) edge[transform canvas={yshift=-2.5pt}, below]  node {$\mathfrak e_1$} (q_0)
 (q_0) edge[loop left] node{$\zeta$} (q_0) 
 (q_1) edge[loop right] node{$\zeta$} (q_1) 
;
\end{tikzpicture}
\end{center}

\item $\mathfrak{gl}_{2,1},$ a four-dimensional basis: $\zeta \in \big \{{\cal E}_1, {\cal E}_2, {\cal E}_3, \mathfrak{h}_1, \mathfrak{h}_2\big\}$

$ $

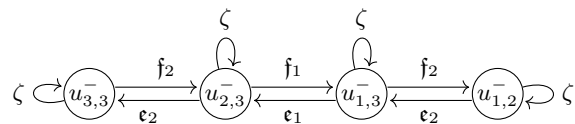
\begin{center}
\footnotesize
\begin{tikzpicture}[shorten >=1pt,node distance=1.8cm,on grid,auto] 
   \node[state,inner sep=0, minimum size=2.0em] (q_0)   {$u^{-}_{3,3}$}; 
   \node[state, inner sep=0, minimum size=2.0em] (q_1) [right=of q_0] {$u^{-}_{2,3}$};    
   \node[state, inner sep=0, minimum size=2.0em] (q_2) [right =of q_1] {$u^{-}_{1,3}$}; 
   \node[state, inner sep=0, minimum size=2.0em] (q_3) [right =of q_2] {$u^{-}_{1,2}$}; 
   \path[->] 
 (q_0)  edge[transform canvas={yshift=2.5pt}, above]  node {$~~\mathfrak f_2$} (q_1)
 (q_1) edge[transform canvas={yshift=-2.5pt}, below]  node {$\mathfrak e_2~~$} (q_0)
 (q_1) edge[transform canvas={yshift=2.5pt}, above]  node {$\mathfrak f_1$} (q_2)
 (q_2) edge[transform canvas={yshift=-2.5pt}, below]  node {$\mathfrak e_1$} (q_1)
 (q_2) edge[transform canvas={yshift=2.5pt}, above] node {$\mathfrak f_2$} (q_3) 
 (q_3) edge[transform canvas={yshift=-2.5pt}, below] node {$\mathfrak e_2$} (q_2) 
 (q_0) edge[loop left] node{$\zeta$} (q_0) 
 (q_1) edge[loop above] node{$\zeta$} (q_1) 
 (q_2) edge[loop above] node{$\zeta$} (q_2) 
(q_3) edge[loop right] node{$\zeta$} (q_3)
;
\end{tikzpicture}
\end{center}
\end{enumerate}
\end{exa}
Investigating the spectral decomposition of the \(\mathfrak{gl}_{k,m}\)-invariant spin-chain Hamiltonian for \(N>2\), following the approach of \cite{mybraided}, represents a natural and vital application of the present findings.

\section{\texorpdfstring{The algebra $\mathfrak{gl}_{1,1}$: irreducible representations}{}}
\noindent In this section, we focus on the study of irreducible representations of the algebra \(\mathfrak{gl}_{1,1}\), 
whose defining algebraic relations are given in Example \ref{gl11}. We note that for the 
\(\mathfrak{gl}_{1,1}\)-invariant \(\check{r}\)-matrix 
(see Proposition \ref{symm0}), the sum 
\(\mathcal{H} = \underset{j \in [N-1]}{\sum} \check{r}_j\) corresponds to 
the Hamiltonian of the Heisenberg XX model in the presence of an external longitudinal magnetic field and special open boundary conditions (see \cite{XX}, as well as \cite{XX2} for a recent related study). 
Indeed, in this case, the \(\check{r}\)-matrix can be written explicitly as \(\check{r}=\frac{1}{2}\big(\sigma ^{x}\otimes \sigma ^{x}+\sigma ^{y}\otimes \sigma ^{y}+\sigma ^{z}\otimes \>\mathrm{id}\>+\>\mathrm{id}\>\otimes \sigma ^{z}\big),\) where \(\sigma^x\), \(\sigma^y \), \(\sigma^z \) are the standard $2 \times 2$ Pauli matrices. Recall also the index notation: for any \(A \in \operatorname{End}(\mathbb{C}^n)\), we define \(A_j := \operatorname{id}^{\otimes(j-1)} \otimes A \otimes \operatorname{id}^{\otimes (N-j)}\). The corresponding XX Hamiltonian is then expressed as: $$
    {\cal H} = \frac{1}{2}\underset{j \in [N-1]}{\sum}\big (\sigma_j^x \sigma^x_{j+1} + \sigma_j^y \sigma^y_{j+1}  + 2 \sigma_j^z \big ) - \frac{1}{2} \big (\sigma_1^z - \sigma_N^z\big ) \in \mbox{End}(({\mathbb C}^2)^{\otimes N}). $$

\begin{lemma} \label{funda0} ({\bf Highest-weight representations}) Let $\varphi: \mathfrak{gl}_{1,1} \to \EEnd(V),$ such that 
$$\epsilon_j \mapsto \varepsilon_j, ~~j \in \big \{1,2\big \}, ~~~ h \mapsto \mathrm{h}, ~~~e \mapsto \mathrm{e}, ~~~f \mapsto \mathrm{f}$$ and
for some $u \in V,$ $j\in\big  \{1,2\big \}:$
$ \varepsilon_j u = \lambda_j u,$ $ ~{\mathrm h} u = \xi u$ and  $~{\mathrm f} u=0,$
$~\lambda_j,\ \xi \in {\mathbb C}.$ Let also $w: = {\mathrm e} u.$ Then, 
\begin{eqnarray}
 \mathrm{e} w =0, ~~{\mathrm h} w = -\xi w,~~\varepsilon_1 w = (\lambda_1 -1) w, ~~\varepsilon_2 w = -(\lambda_2 - \xi) w, ~~\mathrm{f} w = (\lambda_1\xi -\lambda_2) u.\nonumber 
\end{eqnarray}
\end{lemma}
\begin{proof}
The proof is straightforward and is based on the algebraic relation of $\mathfrak{gl}_{1,1}$ (Example \ref{gl11}).
\end{proof}
We conclude that the non-trivial highest-weight representations of $\mathfrak{gl}_{1,1}$ considered in Lemma \ref{funda0} are two-dimensional;  $ \big \{u,\  w  \big \},$ where $w ={\mathfrak e} u,$ is the basis of the corresponding two-dimensional vector space.

We next introduce combinatorial bases of irreducible representations of $\mathfrak{gl}_{1,1}.$ We first establish a convenient notation. 

\begin{rem} \label{not2}
We denote by $\varpi_p^{(N,p)}\in({\mathbb C}^2)^{\otimes N},$ $p \in  \big \{0,1, \ldots, N-1 \big \}$ a linear combination of all possible permutations of $\underbrace{e_1 \otimes \ldots \otimes e_1}_{N-p} \otimes \underbrace{e_2 \otimes \ldots \otimes e_2}_{p}.$
\end{rem}
\begin{pro}  ({\bf Combinatorial bases}) \label{fund1} 
Let $\varpi_p^{(N,p)} \in ({\mathbb C}^2)^{\otimes N},$ $p \in \big \{0,1, \ldots, N-1\big \}$ (introduced in Remark \ref{not2}) be such that:
$~{\mathfrak f} \varpi^{(N,p)}_p =0$ (recall $\mathfrak{e}, \mathfrak {f},$ $\mathfrak{h}, {\cal E}_{1}, {\cal E}_2$ are defined in Remark \ref{nota1}, Example \ref{gl11}). Let also $\varpi^{(N,p)}_{p+1}: = {\mathfrak e} \varpi^{(N,p)}_{p},$ then:
\begin{enumerate}
    
\item ${\cal E}_1 \varpi^{(N,p)}_l = (N-l) \varpi^{(N,p)}_l ,$ $~{\cal E}_2  \varpi^{(N,p)}_l = l(-1)^{l-1} \varpi^{(N,p)}_l,$ $~\mathfrak{h}\ \varpi^{(N,p)}_l = (-1)^l \varpi^{(N,p)}_l,$\\ $l\in \big \{p,\ p+1 \big \},$
$~{\mathfrak e} \varpi^{(N,p)}_{p+1} = 0,$ $~{\mathfrak f} \varpi^{(N,p)}_{p+1} = (-1)^pN \varpi^{(N,p)}_{p}.$ 
\item 
$\varpi^{(N,p)}_{p+1} \perp \varpi^{(N, p+1)}_{p+1}.$\end{enumerate}
\end{pro}
\begin{proof}

$ $

\begin{enumerate}
    \item The first three equations in part (1) are an immediate consequence of the structure of the states $\varpi^{(N,p)}_{l},$ $~ l\in \big \{p,\ p+1 \big \}$  and Remark \ref{nota1}; moreover, ${\mathfrak e} \varpi^{(N,p)}_{p+1} = 0,$ due to ${\mathfrak e}^2 =0.$ The last equation is proven using the algebraic relations of $\mathfrak{gl}_{1,1},$ in particular $\big [{\mathfrak f},\ \mathfrak{e} \big ] = {\cal E_1}{\mathfrak h} - {\cal E}_2$ and the fact that $ \mathfrak{f} \varpi^{(N,p)}_p =0$ and $\varpi^{(N,p)}_{p+1} = \mathfrak{e} \varpi^{(N,p)}_p. $ 

    \item We first recall the standard inner product in $({\mathbb C}^2)^{\otimes N},$ i.e. for every $a,\ b \in ({\mathbb C}^2)^{\otimes N}$ the inner product is defined as $ \langle a,\ b\rangle:= a^{\dagger} \cdot b,$ where $^{\dagger},$ denotes complex conjugation and transposition. Observe also (see Remark \ref{nota1}) that ${\mathfrak e}^T = {\mathfrak h} {\mathfrak f}$ (${\mathfrak e}^T = {\mathfrak e}^{\dagger}$). Then,
   $
    \langle \varpi^{(N,p)}_{p+1},\ \varpi^{(N,p+1)}_{p+1} \rangle =  \langle {\mathfrak e}\ \varpi^{(N,p)}_{p},\ \varpi^{(N,p+1)}_{p+1} \rangle
     = \langle \varpi^{(N,p)}_{p},\  {\mathfrak h}  {\mathfrak f}\ \varpi^{(N,p+1)}_{p+1}  \rangle =0. $
    \qedhere
\end{enumerate}
\end{proof}
Note also that by construction $\varpi_p^{(N,p)} \perp \varpi^{(N,p)}_{p+1}$, we then conclude from Proposition \ref{fund1} that $ \big \{ \varpi^{(N,p)}_p,\ 
\varpi^{(N,p)}_{p+1} \big \}$ is a combinatorial orthogonal basis of a two-dimensional irreducible representation of $\mathfrak{gl}_{1,1}.$

\begin{pro} \label{fund2}
Let $V_{N_i,p_i}$ be the two-dimensional vector space with a basis\\ 
$\big \{\varpi^{(N_i,p_i)}_{p_i},\ \varpi^{(N_i,p_i)}_{p_i+1}\big \},$ $ p_i \in \big \{0,1, \ldots ,N_i-1 \big \},$  $~i \in \big \{1,\ 2 \big \},$ as derived in Proposition \ref{fund1}. Then, 
$$V_{N_1,p_1} \otimes V_{N_2, p_2} = V_{N,p} \oplus V_{N, p+1}, $$ 
where $N= N_1 +N_2$ and $p =p_1 + p_2.$ 
\end{pro}
\begin{proof}

Let,
\begin{enumerate}[{(a)}]
\item $\varpi^{(N,p)}_p: = \varpi^{(N_1,p_1)}_{p_1} \otimes \varpi^{(N_2,p_2)}_{p_2}$ and
$\varpi^{(N,p)}_{p+1}: = {\mathfrak e} \varpi^{(N,p)}_p.$

\item 
$ \varpi^{(N,p+1)}_{p+1}:  =  \varpi^{(N_1,p_1)}_{p_1+1} \otimes  \varpi^{(N_2,p_2)}_{p_2}  - \frac{N_{1}}{N_{2}} (-1)^{p_1}  \varpi^{(N_1,p_1)}_{p_1} \otimes \varpi^{(N_2,p_2)}_{p_2+1}$ and
$ \varpi^{(N,p+1)}_{p+2}:  = \mathfrak{e}  \varpi^{(N,p+1)}_{p+1}.$
\end{enumerate}

We then show that 
$ ~\mathfrak{f}\varpi^{(N,q)}_q =0, ~~q\in \big \{  p,\ p+1 \big \}.$\\
Also, from Proposition \ref{fund1} (recall also Remark \ref{nota1}, part (1)) it follows that relations of part (1) of Proposition \ref{fund1} also hold for $\varpi^{(N,p)}_{p},\ \varpi^{(N,p)}_{p+1}$ and $\varpi^{(N,p+1)}_{p+1},\ \varpi^{(N,p+1)}_{p+2}$ defined in (a) and (b) above.

That is, $V_{N_1,p_1} \otimes V_{N_2, p_2} = V_{N,p} \oplus V_{N, p+1},$ where each $V_{N,q},$ $q\in \big\{p,p+1  \big \}$ is a 
two-dimensional vector space with a basis $\big \{ \varpi^{(N,q)}_q,\ \varpi^{(N,q)}_{q+1}\}.$
\end{proof}

\subsection*{Irreducible representations as Young tableaux}
We recall basic definitions of Young tableaux as these are essential combinatorial objects that play a central role in representation theory  \cite{FultonHarris}.
We denote $\lambda \vdash N$ a partition $\lambda= (\lambda_1, \lambda_2, \ldots,\lambda_p)$ of the positive integer $N$,
where $\lambda_i$ are weakly decreasing positive integers and $\underset{i\in[p]}{\sum} \lambda_i = N.$ The size of $\lambda$ is denoted $|\lambda|,$ and in general $|\lambda| = N.$ 

\begin{defn} ({\bf Young diagram}) \label{Young} Suppose $\lambda= (\lambda_1, \lambda_2, \ldots,\lambda_p),$ $\lambda \vdash N$ where $p\geq 1$. The Young (or Ferrers) diagram of
shape $\lambda$ is an array of $N$ squares having $p$ rows with row $i$ containing $\lambda_i$
squares.
\end{defn}
\begin{defn}
    A filling (or weight) of a Young diagram is any way of putting a positive integer
in each box of the diagram. Let $\mu = (\mu_1, \mu_2, . . . , \mu_l)$ be a filling of a Young diagram.
Each $\mu_i$ is the number of times the integer $i$ appears in the diagram.
\end{defn}
In order for the diagram to be completely filled, it is necessary for $|\lambda| = |\mu|$. 
It is possible to fill diagrams arbitrarily in this manner, however 
we impose certain restrictions on the filling $\mu$. These restrictions lead to
the definition of a Young tableau.
\begin{defn} \label{Young2} ({\bf Standard Young tableaux})
Suppose $\lambda \vdash N.$ A Young tableau $T$ is obtained by filling in the boxes of the Young 
diagram with symbols taken from some alphabet, which is usually required to be a totally ordered set. 
A Young tableau of shape $\lambda$ is  called a $\lambda$-tableau. A Young tableau is standard 
if the rows and columns of $T$ are
increasing sequences. That is, $T$ is filled with the numbers $1, 2, \ldots,N$ bijectively. 
\end{defn}
We consider the set $[n]$ with the standard ordering $1<2< \ldots k<k+1< \ldots <n.$
There are various definitions for semi-standard Young tableaux depending on the variation 
of the associated Schur functions \cite{Mac92}.
We use here the following definition of a semi-standard Young tableau.
\begin{defn} \label{ssyt}
A Young tableau is semi-standard if the filling is:
\begin{enumerate}
    \item 
weakly increasing across each row and strictly increasing down each column for numbers $\big \{1,2, \ldots, k \big\}$ 
\item strictly increasing across each row and weakly increasing down each column for numbers $\big \{k+1,k+2, \ldots, n\big \}.$ 
\end{enumerate}
\end{defn}
Henceforth, we use the shorthand notation $SSYT$ and $SYT$ for semi-standard and standard Young tableau respectively.
The $SSYT$ defined above are associated to the hook Schur functions also known as super-symmetric Schur functions and correspond to representations of the Lie superalgebra $\mathfrak{gl}(k|m)$, see \cite{Kac}. The hook Schur functions were introduced in \cite{BR83} and correspond to the sixth variant of Schur functions considered in \cite{Mac92}. 

We focus now on the algebra $\mathfrak{gl}_{1,1}$. In this case the Young tableaux are filled by 1 and 2, and according to Definition \ref{ssyt} the only allowed $SSYTs$ are of shape $\lambda = (N-p, \underbrace{1, 1, \ldots, 1}_{p})$, with two possible fillings: $\mu_1= (N-p,\ p)$ and $\mu_2= (N-p-1,\ p+1).$  The two-element set of $SSYTs$ of shape $\lambda = (N-p, \underbrace{1,1, \ldots,1}_p)$ is denoted $SSYT(N, p).$

Each one of the two elements in $SSYT(N,p)$ corresponds to an element of the basis ${\mathrm B}_{N,p}=\big \{ \varpi_p^{(N,p)},\ \varpi^{(N,p)}_{p+1}\big \} $ of the two-dimensional vector space $V_{N,p}$ (see Proposition \ref{fund1}). That is, there is a bijective map between the sets $\big \{{\mathrm B}_{N,p}\big \}$ and $\big \{SSYT(N,p)\big \}$, $p\in \big \{0,1, \ldots, N-1\big \},$ such that ${\mathrm B}_{N,p} \mapsto SSYT(N,p),$ specifically\\ 
\begin{center}
$\varpi^{(N,p)}_{p} \mapsto$  {\tiny $\underbrace{\begin{ytableau}
\ & & & \ldots & & \\
\  \bullet  \\ 
\ \vdots\\
\ \bullet \\
\ \bullet
\end{ytableau}}_{(N-p)-boxes},~~~$} $ ~~~~\varpi^{(N,p)}_{p+1} \mapsto$ {\tiny $\underbrace{\begin{ytableau}
\ & & & \ldots & & \bullet \\
\ \bullet  \\ 
\ \vdots\\
\ \bullet \\
\ \bullet
\end{ytableau}}_{(N-p)-boxes}.$} 
\end{center}
The Young tableaux above correspond to the fillings $\mu_1= (N-p,\ p)$ and $\mu_2= (N-p-1,\ p+1)$ respectively. The empty boxes in the tableaux are occupied by 1, whereas the boxes containing a bullet are occupied by 2.

%

\end{document}